\documentclass[12pt]{article}

\usepackage{sola}
\usepackage{todonotes}

\usepackage{pgf,ytableau}
\usepackage{hyperref}
\usepackage{algorithm}
\usepackage[noend]{algpseudocode}

\usepackage[mathrm=sym,warnings-off={mathtools-colon,mathtools-overbracket}]{unicode-math}

\theoremstyle{definition}
\newtheorem{remark}{Remark}

\DeclareMathOperator{\sign}{sign}
\DeclareMathOperator{\curl}{curl}

\DeclareMathOperator{\divergence}{div}
\DeclareMathOperator{\mathspan}{span}

\DeclareMathOperator{\mdiag}{diag}

\def\nablaperp{\nabla_\perp}

\def\iip<#1,#2>{\left⟪ #1, #2 \right⟫}

\def\vc#1{{\bf #1}}

\def\bbR{{\mathbb R}}
\def\bbZ{{\mathbb Z}}
\def\bbC{{\mathbb C}}

\def\calM{{\mathcal M}}
\def\calD{{\mathcal D}}

\def\calK{{\mathcal K}}

\def\calN{{\mathcal N}}
\def\calT{{\mathcal T}}
\def\calE{{\mathcal E}}
\def\bbT{{\mathbb T}}
\def\bfe{\vc e}
\def\bfq{\vc q}

\def\bfa{\vc a}
\def\bfb{\vc b}
\def\bfx{\vc x}

\def\bfy{\vc y}
\def\bfz{\vc z}
\def\bfP{\vc P}
\def\bfQ{\vc Q}
\def\rmy{{\mathrm y}}
\def\rmz{{\mathrm z}}

\def\rmw{{\mathrm w}}

\def\rmZ{{\mathrm Z}}
\def\rmY{{\mathrm Y}}

\def\bfp{\vc p}
\def\bff{\vc f}
\def\bfg{\vc g}
\def\bfc{\vc c}
\def\bfv{\vc v}
\def\bfn{\vc n}
\def\bft{\vc t}

\def\calR{{\mathcal R}}
\def\calP{{\mathcal P}}

\def\calB{{\mathcal B}}

\def\D{\hbox{d}}

\def\ipF<#1>{\ip<#1>_{\mathrm F}}

\def\I{\mathrm{i}}

\def\Ho{\mathring{H}}

\def\Hocurl{\mathring{H}_{\mathrm{curl}}}
\def\Hodiv{\mathring{H}_{\mathrm{div}}}
\def\Hcurl{H_{\mathrm{curl}}}

\def\Hdiv{H_{\mathrm{div}}}

\def\weirdrat{𝔯_n^{(b)}}
\def\hexnumber#1{\ifcase#1 0\or1\or2\or3\or4\or5\or6\or7\or8\or9\or
A\or B\or C\or D\or E\or F\fi}
\edef\msbhx{\hexnumber\symAMSb}   
\mathchardef\emptyset="0\msbhx3F

\begin{document}

\setupproofs
\setupenvironments



\refmodetrue

\authord={Sheehan Olver\footnote{Department of Mathematics,  Imperial College London, UK, {\tt s.olver@imperial.ac.uk}}}

\titled={Orthogonal polynomials for the de Rham complex on the disk and cylinder}

\maketitle

\Abstract
This paper constructs polynomial bases that capture the structure of the de Rham complex with boundary conditions in  disks and cylinders (both periodic and finite) in a way that respects  rotational symmetry.
The starting point is explicit constructions of vector and matrix orthogonal polynomials on the unit disk that are analogous to the (scalar) generalised Zernike polynomials. We use these to build new orthogonal polynomials with respect to a matrix weight that forces vector polynomials to be normal on the boundary of the disk. The resulting weighted vector orthogonal polynomials  have a simple connection to the gradient of weighted generalised Zernike polynomials, and  their  curl (i.e. vorticity or rot) is a constant multiple of the standard Zernike polynomials which are orthogonal with respect to $L^2$ on the disk. This construction naturally leads to bases in cylinders with simple recurrences relating their gradient, curl and divergence. These bases decouple the de Rham complex into small exact sub-complexes.

\Section{intro} Introduction.

Multivariate orthogonal polynomials (OPs) with respect to an ultraspherical-like inner product on the unit disk
\[
\ip<f, g>_λ := \iint_Ω f(\bfx) g(\bfx) (1-r^2)^λ \D \bfx
\]
for $Ω := \{ \bfx = \vectt[x,y] \in \bbR^2 : \| \bfx \| ≤ 1\}$ are given by the generalised Zernike polynomials
\begin{equation}\label{Equation:Zernike}
\rmz_{mj}^{(λ)}(\bfx) := P_j^{(λ,|m|)}(2r^2-1)\rmy_m(\bfx)
\end{equation}
for the basis of harmonic polynomials
\begin{equation}\label{Equation:harmonic}
\rmy_m(\bfx) = r^{|m|} \E^{\I m θ} = (x + \sign m \I y)^{|m|}
\end{equation}
where $r = \sqrt{x^2 + y^2}$,  $x = r \cos θ$, $y = r \sin θ$, and $P_k^{(a,b)}$ are the Jacobi polynomials, see eg. \cite[Section 18]{DLMF} for their definition.
The Zernike polynomials $\rmz_{mj} := \rmz_{mj}^{(0)}$ were derived in \cite{Zernike} but more recently their generalised form has proven useful for solving partial differential equations in disks \cite{vasil2016tensor}
as they lead naturally to sparse recurrence relationships involving differential operators such as Laplacians
and gradients. They also have nice relationships for power-law kernels \cite{gutleb2023computation} and fractional
Laplacians \cite{papadopoulos2025frame}. Weighted Zernike polynomials
\begin{equation}\label{Equation:Bubbles}
\rmw_{mj}(x,y) := (1-r^2) z_{mj}^{(1)}(x,y).
\end{equation}
are also Sobolev orthogonal polynomials on a disk (or more generally, a ball), see, eg. \cite{xu2006family,xu2008sobolev,perez2013weighted,figueroa2017orthogonal,marriaga2023approximation,figueroa2023weighted}.

This paper considers vector and matrix analogues of Zernike polynomials with  ultraspherical-like inner products leading to simple recurrence relationships for gradients
and a two-dimensional version of curl. In particular, the gradient of  weighted Zernike basis $\rmw_{mj}$ has a simple expression in terms of  a new family of weighted vector orthogonal polynomials $\bfn_{mj}^ν$, which have the property that they are normal on the boundary of the disk (see \thref{gradW}). We depict some low order examples of $\bfn_{mj}^ν$ in \figref{normalpolys}. In turn, the curl (sometimes called the vorticity or rot) of $\bfn_{mj}^ν$ is up-to-a-constant equal to the standard Zernike polynomials $\rmz_{mj}$ (see \thref{curlW}). Thus these are natural bases for the de Rham complex on a disk, and lay the groundwork for incorporation into a Finite Element Exterior Calculus (FEEC) framework
 for solving partial differential equations like the Hodge–Laplacian via a mixed weak formulation or Maxwell's equation, see the review in \cite{Arnold2018}. But the primary goal of this paper is to introduce these new OPs and
their beautiful properties.

An important component of our construction is that the bases behave like Fourier series in the sense that rotations correspond to simple equivariant transformations. This is a very simple example of a {\it symmetry-adapted basis}, see \cite{fassler1976application,stiefel2012group,olver2024parallelisation}, which are bases associated with symmetry groups where the  group action corresponds to multiplication by an irreducible representation. The connection between representation theory and FEEC has been explored recently in the case of simplicial meshes \cite{licht2024symmetry,berchenko2024symmetric}.
In the Fourier case the irreducible representations of a rotation by angle $φ$ are simply multiplication by a scalar $\E^{\I m φ}$. Indeed,  a rotation by $φ$ applied to the arguments $\rmw_{mj}$ and $\rmz_{mj}$ is  a multiplication by $\E^{\I m φ}$ hence these correspond to scalar symmetry-adapted bases. Applying a rotation to the arguments of $\bfn_{mj}^ν$ is multiplication by $\E^{\I m φ}$ combined with a rotation of the vector, an example of a vector symmetry-adapted basis. We
build on the language of symmetry-adapted bases to facilitate future generalisation to balls, where the irreducible representations are matrix-valued, and induced by spherical harmonics. In particular, we repeatedly use the rather trivial result that if $a \in \bbC$ satisfies
\[
\E^{\I m φ}a = a \E^{\I n φ}
\]
for all rotations $φ \in \bbR$ with $m ≠ n$ then $a = 0$. This is a baby version of Schur's lemma (see e.g. \cite[Schur's Lemma 1.7]{FultonRepTheory}), a connection that will be key to future generalisation. Note that Schur's lemma requires irreducible representations over $\bbC$, and hence the use of complex-valued orthogonal polynomials is essential.

\Figurew[normalpolys]{\hsize}
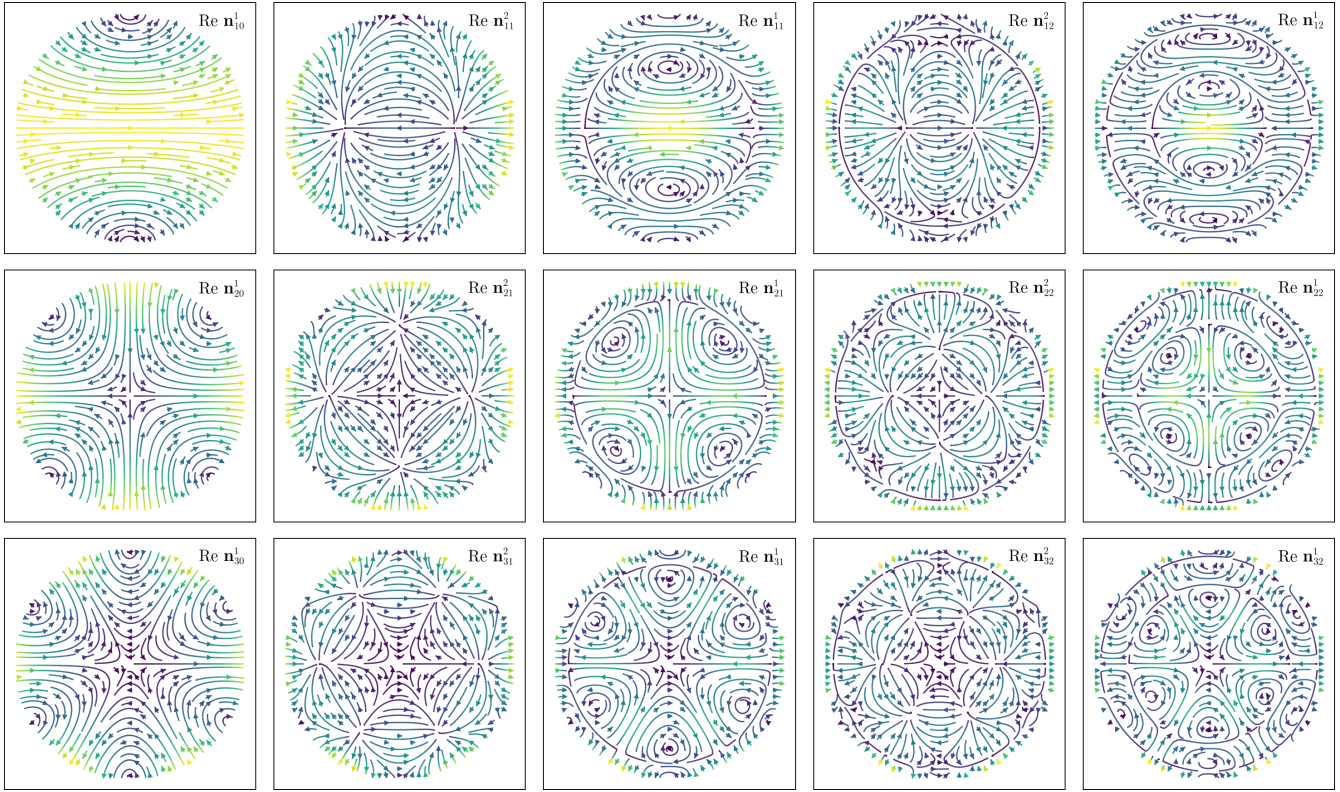
	The real part of some vector polynomials $\bfn_{mj}^ν$, which span all polynomials that are normal on the boundary of the disk. These are a symmetry-adapted basis and the parameter $m$ dictates that it behaves like the corresponding Fourier mode, and the rotational symmetry is evident.  The gradient of weighted generalised Zernike polynomials  $\{\rmw_{mj}\}$ has a simple expression in terms of $\{\bfn_{mj}^ν\}$, and their curl is a constant multiple of the standard Zernike polynomials $\{\rmz_{mj}\}$.


The structure of the paper is as follows:

\secref{symmetryadapted}: We review the notions of invariance and equivariance for rotations and its relationship with vector calculus operators. We also introduce the concept of scalar, vector and matrix symmetry-adapted bases, which in the case of rotations transform under rotations through simple relationships.

\secref{homo}: We give explicit constructions of scalar, vector and matrix symmetry-adapted homogeneous polynomials with respect to rotations.

\secref{ops}: We give explicit constructions of scalar, vector and matrix symmetry-adapted OPs with respect to an ultraspherical-like scalar weights $(1-r^2)^λ$ in terms of Jacobi polynomials.

\secref{vectoruniOPs}: To go beyond scalar weights we will relate multivariate vector OPs to a special class of univariate vector OPs living in a specific module of polynomials. We discuss their properties, and construct vector OPs explicitly for a weight that will lead to the construction of $\bfn_{mj}^ν$. To construct such OPs we adapt the idea of using Cholesky factorisations introduced by Gautschi \cite{gautschi1970construction}, see also \cite{gutleb2024polynomial}, but in this instance the Cholesky factorisation leads to explicit expressions in terms of Jacobi polynomials.

\secref{matrixweights}: We relate univariate vector OPs constructed in \secref{vectoruniOPs} to  multivariate vector OPs with respect to a class of matrix-valued polynomial weights. This connection is used in \secref{normal}  to construct $\bfn_{mj}^ν$, a complete basis of polynomials that are normal on the boundary of the disk.

\secref{veccalc}: We arrive at the main result of the paper: the weighted scalar OPs $\rmw_{mj}$, vector OPs $\bfn_{mj}^ν$ and scalar OPs $\rmz_{mj}$ are a natural basis for the 2D de Rham complex with boundary conditions on the disk. In particular, taking a gradient of weighted Zernike polynomials $\rmw_{mj}$ gives a simple expression in terms of $\bfn_{mj}^ν$. In turn, their curl gives a simple expression in terms of Zernike polynomials $\rmz_{mj}$. These bases can be used to decompose the de Rham complex into simple sub-complexes.

\secref{cylinders}: We see that the results extend naturally to the 3D de Rham complex with boundary conditions on  both periodic cylinders (using Fourier series in the $z$ direction) and finite cylinders (using weighted ultraspherical polynomials, i.e., integrated Legendre polynomials \ala\ \cite{babuska1981p}).

\secref{conc}: We conclude by discussing how  these results can be used to solve partial differential equations like Maxwell's equation in a cylinder.
We also discuss the possibility of extensions to other complexes such as the Koszul complex and the elasticity complex.

\begin{table}[tb]
\centering
\caption{Notations.\label{Table:Notation}}
\[
\begin{array}{c| c}
	\hline\\[-3mm]
	ρ(φ) & \hbox{A matrix corresponding to rotation by $φ$, see \eqref{Equation:rotation}.} \hfill\\
	\calR_φ & \hbox{The linear operator corresponding to rotating  variables, see \eqref{Equation:rotationop}.} \hfill\\
Σ(x,y) & \hbox{$\sopmatrix{
x^2-y^2 & 2xy \\
 2xy  & y^2-x^2
}$, see \eqref{Equation:G}.}\hfill\\
N(x,y) & \hbox{$\sopmatrix{
1 - y^2 & xy \\
x y & 1-x^2
}$, see \eqref{Equation:N}, which causes vectors to be normal at the boundary.} \hfill\\
T(x,y) & \hbox{$\sopmatrix{
1 - x^2 & -xy \\
-x y & 1-y^2
}$, see \eqref{Equation:T}, which causes vectors to be tangential at the boundary.} \hfill\\
\calM & \hbox{A special module of vector polynomials, see \eqref{Equation:calM}.}\hfill\\
\calP_m & \hbox{Map between $\calM$ and symmetry-adapted vector polynomials, see \lmref{Pm}.}\hfill\\
\calN & \hbox{The space of vector polynomials normal to the boundary of the disk, see \eqref{Equation:calN}.} \hfill\\
\weirdrat & \hbox{The constant $10n^2 + n(13b+19)+ 4(b+1)(b+2)$, see \defref{normalop}.}\hfill\\
κ_{mj} & \hbox{A constant arising in the recurrence relationship for curl, see \eqref{Equation:kappa}.}\hfill\\[2mm]
\hline
\end{array}
\]
\end{table}

\begin{table}[tb]
\centering
\caption{Different types of polynomials.\label{Table:polys}}
\[
\begin{array}{c| c}
	\hline\\[-3mm]
\rmy_m(x,y) & \hbox{Harmonic polynomials, see \eqref{Equation:harmonic}.} \hfill\\
\bfy_m(x,y) & \hbox{A vector analogue of harmonic polynomials, see \defref{vectorharmonicpolys}.} \hfill\\
\rmY_m(x,y) & \hbox{A matrix analogue of harmonic polynomials, see \defref{matrixharmonicpolys}.} \hfill\\
\rmz_{mj}^{(λ)}(x,y) & \hbox{Generalised Zernike polynomials orthogonal to $(1-r^2)^λ$, see \eqref{Equation:Zernike}.} \hfill\\
\rmw_{mj}(x,y) & \hbox{Weighted generalised Zernike polynomials $(1-r^2) \rmz_{mj}^{(1)}(x,y)$, see \eqref{Equation:Bubbles}.} \hfill\\
\bfz_{mj}^{(λ),ν}(x,y) & \hbox{Vector Zernike polynomials orthogonal to $(1-r^2)^λ$, see \defref{vectorZernike}.} \hfill\\
\rmZ_{mj}^{(λ),ν}(x,y) & \hbox{Matrix Zernike polynomials orthogonal to $(1-r^2)^λ$, see \defref{matrixZernike}.} \hfill\\
\bfp_n^{V,ν}(t) & \hbox{Polynomials in $\calM$ orthogonal with respect to a matrix weight $V$, see \eqref{Equation:OPsinM}.} \hfill\\
\bfp_n^{(a,b),ν}(t) & \hbox{Polynomials in $\calM$ orthogonal with respect to $(1-t)^a t^b$, see \propref{jacobiWt}.} \hfill\\
\bfq_n^{(b),ν}(t) & \hbox{Polynomials in $\calM$ orthogonal with respect to $\sopmatrix{1 \hfill\\ & 1-t} t^b$, see \defref{normalop}.} \hfill\\
\bfv_{mj}^{W,ν}(x,y) & \hbox{Vector polynomials orthogonal with respect to a matrix weight $W$, see \thref{rhoGortho}.} \hfill\\
\bfn_{mj}^ν(x,y) & \hbox{A basis for $\calN$ of weighted vector OPs, see \defref{normalops}.} \hfill\\
\bfn_{mj}^±(x,y) & \hbox{A simple recombination of $\bfn_{mj}^1$ and $\bfn_{mj}^2$, see \defref{nplus}.}\hfill\\
\bft_{mj}^ν(x,y) & \hbox{A basis for polynomials tangent on the boundary of a disk, see \eqref{Equation:tangent}.} \hfill\\[2mm]
\hline
\end{array}
\]
\end{table}

In Table~\ref{Table:Notation} we give the notations used in this paper.
In Table~\ref{Table:polys} we list the different types of polynomials we will use.

\medskip

\noindent{\bf Acknowledgements}:  I thank Doug Arnold, Kaibo Hu, Christoph Ortner, Vic Reiner, Lior Silberman,  Alex Townsend, Heather Wilber and Grady Wright for helpful discussions and suggestions.

\Section{symmetryadapted} Invariance, equivariance, and vector calculus.

In this section we discuss  invariance and equivariance with respect to rotations and its relationship with vector calculus and symmetry-adapted bases. We denote rotation matrices by
\begin{equation}\label{Equation:rotation}
ρ(φ) := \sopmatrix{ \cos φ & -\sin φ \\
\sin φ & \cos φ}
\end{equation}
and
the linear operator corresponding to a rotating the variables by $φ$ as
\begin{equation}\label{Equation:rotationop}
\calR_φ u(\bfx) := u(ρ(φ) \bfx).
\end{equation}

\subsection{Invariance and Equivariance}\label{Section:InvOrEqui}

An important notion is invariance to rotations:

\Definition{invariance}
A scalar function $f : Ω \rightarrow \bbC$ is invariant (to rotations) if it satisfies:
\[
\calR_φ f(\bfx) = f(\bfx).
\]
In other words,
\[
f(ρ(φ) \bfx) = f(\bfx).
\]

To extend this notion to vector-valued functions we need to incorporate a notion of equivariance: rather than functions not changing under rotations, the analogue are functions where the vector output is also rotated.
It is natural to view vector-valued functions as maps $Ω \rightarrow \bbC^2$, that is to a Euclidean vector space.
But the notion of equivariance is more natural when discussed in terms of tangent spaces.

\Definition{tangent}
Denote the (complex) {\it tangent space} at $\bfx \in Ω$ as:
\[
\calT_\bfx :=  \{ a \bfe_r^{\bfx} + b \bfe_θ^{\bfx} : a,b \in \bbC\}.
\]
for the two orthogonal vectors
\begin{equation}\label{Equation:er}
\bfe_r^{\vc x} := {1 \over r} \Vectt[x,y] = \Vectt[\cos θ, \sin θ],\qquad \bfe_θ^{\vc x} := {1 \over r}  \Vectt[-y,x] = \Vectt[-\sin θ, \cos θ]
\end{equation}
where $r = \sqrt{x^2+y^2}$, $x = r \cos θ$ and $y = r \sin θ$. When $\bfx$  is clear from context we write $\bfe_r$ and $\bfe_θ$.

Rotations map these vectors via the rule:
\begin{equation}
ρ(φ) \bfe_r^{\vc x} = \calR_φ \bfe_r^{\bfx} = \bfe_r^{ρ(φ)\vc x}, \qquad  ρ(φ) \bfe_θ^{\vc x}  = \calR_φ \bfe_θ^{\bfx} =  \bfe_θ^{ρ(φ)\vc x}. \label{Equation:rotateorthovecs}
\end{equation}
Thus we can interpret multiplying a vector by $ρ(φ)$ as a map from an element in the tangent space of $\bfx$ to
the tangent space of $ρ(φ) \bfx$. More precisely, rotation induces a map $ρ(φ) : \calT_\bfx \rightarrow \calT_{ρ(φ) \bfx}$ via the formula
\[
ρ(φ) \underbrace{( a \bfe_r^{\bfx } + b \bfe_θ^{\bfx})}_{\in \calT_{\bfx}} =   \underbrace{a \bfe_r^{ρ(φ)\bfx } + b \bfe_θ^{ρ(φ)\bfx}}_{\in \calT_{ρ(φ)\bfx}}.
\]
A vector-valued function is equivariant if we have
\[
\underbrace{\calR_φ 𝐟(\bfx)}_{\in \calT_{ρ(φ) \bfx}} = \underbrace{ρ(φ)}_{\calT_\bfx \rightarrow \calT_{ρ(φ)\bfx}} \underbrace{𝐟(\bfx)}_{\in \calT_{\bfx}}.
\]
But in this case $\calT_\bfx = \calT_{ρ(φ)\bfx} =  \bbC^2$ so we can re-express this relationship without using tangent spaces\footnote{Future extensions to spherical caps \ala\ \cite{snowball2021sparse} will require tangent spaces.}:

\Definition{vectorequivariance}
A vector-valued function $𝐟 : Ω \rightarrow \bbC^2$ is equivariant (to rotations) if it satisfies:
\[
\calR_φ 𝐟 = ρ(φ) 𝐟.
\]
In other words,
\[
𝐟(ρ(φ) \bfx) = ρ(φ) 𝐟(\bfx).
\]

Note that  a constant vector-valued function like $\bfx \mapsto \vectt[1,0]$ is invariant to rotations but it is {\it not equivariant}.
On the other hand,  the functions $\bfx \mapsto \bfx$ and $\bfx \mapsto \vectt[-y,x]$ are equivariant.

We can also interpret matrix-valued functions $F : Ω \rightarrow \bbC^{2 \times 2}$ as linear operators acting on tangent spaces, i.e., $F(\bfx) : \calT_{\bfx} \rightarrow \calT_{\bfx}$.  With this in mind we  extend the notion of equivarance  to matrices by ensuring that we transform the space they act on in an appropriate way:
\[
\underbrace{F(ρ(φ) \bfx)}_{\calT_{ρ(φ)\bfx} \rightarrow \calT_{ρ(φ)\bfx}} \quad\underbrace{ρ(φ)}_{\calT_{\bfx} \rightarrow \calT_{ρ(φ)\bfx}} =
\underbrace{ρ(φ)}_{\calT_{\bfx} \rightarrow \calT_{ρ(φ)\bfx}}\quad \underbrace{F(\bfx)}_{\calT_{\bfx} \rightarrow \calT_{\bfx}}.
\]
Dropping the use of tangent spaces we have the following:

\Definition{matrixequivariance}
A matrix-valued function $F : Ω \rightarrow \bbC^{2 \times 2}$ is equivariant (to rotations) if it satisfies:
\[
\calR_φ F ρ(φ) = ρ(φ) F.
\]
That is to say
\[
F(ρ(φ) \bfx) ρ(φ) = ρ(φ) F(\bfx).
\]

\begin{remark}
	We will  use right-associativity for operators, so that in the above notation we have
\[
\calR_φ F ρ(φ) = \calR_φ[F ρ(φ)].
\]
Though in this case multiplication by a matrix is independent of the change-of-variables so associativity does not impact the calculation as we have:
\[
\calR_φ[F ρ(φ)](\bfx) = F(ρ(φ)\bfx) ρ(φ) = \calR_φ[F](\bfx) ρ(φ)
\]
\end{remark}

A trivial example of an equivariant matrix-valued functions is the identity $\bfx \mapsto I$.
Indeed, any function of the form $f(r) I$ is equivariant.
Another example which we will use throughout the paper is
\begin{equation}\label{Equation:G}
Σ(x,y) := \sopmatrix{
x^2-y^2 & 2xy\\
 2xy  & y^2-x^2
}.
\end{equation}
To see this is equivariant it helps to express the matrix in terms of $\bfe_r^\bfx$ and $\bfe_θ^\bfx$:
\[
Σ(x,y) = \sopmatrix{x & -y \\
y & x}  \sopmatrix{1 \\ & -1} \sopmatrix{x & y \\
-y & x} = \bvect[\bfe_r^{\bfx},\bfe_θ^{\bfx}] \sopmatrix{r^2 \\ & -r^2} \Vectt[(\bfe_r^{\bfx})^\top,(\bfe_θ^{\bfx})^\top].
\]
This implies that
\meeq{
\calR_φ Σ(\bfx) ρ(φ) =
\bvect[\bfe_r^{ρ(φ)\bfx},\bfe_θ^{ρ(φ)\bfx}] \sopmatrix{r^2 \\ & -r^2} \Vectt[(\bfe_r^{ρ(φ)\bfx})^\top,(\bfe_θ^{ρ(φ)\bfx})^\top] ρ(φ) \ccr
= ρ(φ) \bvect[\bfe_r^{\bfx},\bfe_θ^{\bfx}] \sopmatrix{r^2 \\ & -r^2} \Vectt[(\bfe_r^{\bfx})^\top,(\bfe_θ^{\bfx})^\top]
= ρ(φ)Σ(\bfx).
}

\subsection{Symmetry-adapted bases}

To incorporate rotational symmetry into a basis of functions that are not invariant or equivariant (in the sense used in the previous section) we need
to allow change with rotations, but where this change is equivalent to multiplication\footnote{This could also be called {\it equivariance} but  we reserve the term equivariance for the definitions in \secref{InvOrEqui}.}  by $\E^{\I m φ}$.
These correspond to irreducible representations of the group $SO(2)$, an analogy which motivates the terminology
of what follows, see \cite{olver2024parallelisation}.

\Definition{symmetryadapted}
	A {\it (scalar)  symmetry-adapted function} $p : Ω \rightarrow \bbC$ with mode $m$ is one where  a rotation becomes multiplication:
\[
\calR_φ p = p \E^{\I m φ},
\]
i.e.,
\[
p(ρ(φ) \bfx) = p(\bfx) \E^{\I m φ}.
\]
A {\it (scalar) symmetry-adapted basis} is one whose basis elements are symmetry-adapted functions.

The basis of harmonic polynomials defined in the introduction form a symmetry-adapted basis: $\rmy_m(\bfx)$ is symmetry adapted with mode $m$  since we have
\meeq{
\calR_φ \rmy_m(\bfx) = \rmy_m(
x \cos φ  - y \sin φ,
x \sin φ  + y \cos φ) = r^{|m|} \E^{\I m (θ+φ)} =  \rmy_m(\bfx) \E^{\I m φ}.
}

In the vector case we also rotate the vector in accordance with the principle of equivariance:

\Definition{vectorsymmetryadapted}
	A {\it vector  symmetry-adapted function} $\bfp : Ω \rightarrow \bbC^2$ with mode $m$ satisfies
\[
\calR_φ \bfp= ρ(φ) \bfp \E^{\I m φ},
\]
i.e.,
\[
\bfp(ρ(φ) \bfx) = ρ(φ) \bfp(\bfx) \E^{\I m φ}.
\]
A {\it vector symmetry-adapted basis} is one whose basis elements are vector symmetry-adapted functions.

Note if $\bfp$ is a vector symmetry-adapted function and $G$ is a matrix equivariant function then $G \bfp$ is a vector symmetry-adapted function
with the same mode as $\bfp$ since
\[
ρ(φ) G \bfp \E^{\I m φ} = \calR_φ[G] ρ(φ) \bfp \E^{\I m φ}  = \calR_φ[G]  \calR_φ[\bfp]  = \calR_φ[G\bfp].
\]

We can extend this notion to matrix-valued functions as well by rotating the vectors they act on appropriately:

\Definition{matrixsymmetryadapted}
A {\it matrix  symmetry-adapted function} $P : Ω \rightarrow \bbC^{2\times2}$ with mode $m$ satisfies
\[
\calR_φ P ρ(φ) = ρ(φ) P \E^{\I m φ},
\]
i.e.,
\[
P(ρ(φ) \bfx) ρ(φ) = ρ(φ) P(\bfx) \E^{\I m φ}.
\]
A {\it matrix symmetry-adapted basis} is one whose basis elements are matrix  symmetry-adapted functions.

An important feature of symmetry-adapted bases is that symmetry-adapted functions with different modes are
automatically orthogonal. To avoid issues with regularity we state the results in this paper in terms of {\it symmetry-adapted polynomials}:

\Lemma{ortho}
If $a$ and $b$ are  scalar  symmetry-adapted polynomials with respect to modes $m$ and $n$, respectively, with $m ≠ n$ then $\ip<a, b>_w  = 0$ where, for any invariant weight $w$,
\[
\ip<a, b>_w := \iint_Ω \bar a(\bfx) b(\bfx)  w(\bfx) \D \bfx.
\]
If $\bfa$ and $\bfb$ are  scalar  symmetry-adapted polynomials with respect to modes $m$ and $n$, respectively, with $m ≠ n$ then $\ip<\bfa, \bfb>_W  = 0$ where, for any matrix equivariant weight $W$,
\[
\ip<\bfa, \bfb>_W := \iint_Ω  \bfa(\bfx)^\star W(\bfx) \bfb(\bfx)   \D \bfx.
\]
If $A$ and $B$ are matrix  symmetry-adapted polynomials with respect to modes $m$ and $n$, respectively, with $m ≠ n$  then
$
\iip<A, B>_w = 0
$
where, for any invariant weight $w$, we define the matrix inner product
\[
\iip<A, B>_w := \iint_Ω \ipF<A(\bfx),B(\bfx)>  w(\bfx) \D \bfx,
\]
for the Fröbenius inner product $\ipF<F,G> := \mathrm{Tr}(F^\star G)$.

\Proof

First note that we know the adjoint of a rotation, i.e., for any $a,b$,
\[
\ip<\calR_φ a,b>_w = \iint_Ω \bar a(ρ(φ)\bfx) b(\bfx)  w(r) \D \bfx = \iint_Ω \bar a(\bfx) b(ρ(-φ)\bfx) |\det ρ(φ)| w(r) \D \bfx  = \ip<a, \calR_{-φ} b>_w.
\]
We therefore have for all $φ$
\[
\E^{\I m φ} \ip<a, b>_w =  \ip<\E^{-\I m φ}a, b>_w = \ip<\calR_{-φ} a, b>_w = \ip< a, \calR_φ b>_w
= \ip<a, b>_w \E^{\I n φ}
\]
which shows that $\ip<a, b>_w = 0$ when $m ≠ n$ (which as explained in the introduction is a trivial version of Schur's lemma).
In the vector case, using the equivariance of $W$ in the form $ρ(φ) W(ρ(-φ) \bfx) = W(\bfx) ρ(φ)$, we have for, any $\bfa,\bfb$,
\meeq{
\ip<ρ(-φ) \calR_φ \bfa,\bfb>_W = \iint_Ω  \bfa(ρ(φ)\bfx)^\star ρ(φ) W(\bfx) \bfb(\bfx) \D \bfx \ccr
= \iint_Ω  \bfa(\bfx)^\star ρ(φ) W(ρ(-φ) \bfx) \bfb(ρ(-φ)\bfx) |\!\det ρ(φ)| \D \bfx \ccr
 = \iint_Ω  \bfa(\bfx)^\star W(\bfx) ρ(φ)\bfb(ρ(-φ)\bfx)  \D \bfx  = \ip<\bfa, ρ(φ) \calR_{-φ} \bfb>_W.
}
 It follows that, for all $φ$,
\[
\E^{\I m φ} \ip<\bfa, \bfb>_W = \ip<ρ(φ) \calR_{-φ} \bfa, \bfb>_W = \ip< \bfa, ρ(-φ) \calR_{φ} \bfb>_W
= \ip<\bfa, \bfb>_W \E^{\I n φ},
\]
and hence $\ip<\bfa, \bfb>_W =0$ when $m ≠ n$.

Finally, in the matrix case we first note since the trace is invariant to conjugation by an orthogonal matrix we have (for any matrices $A,B$)
\meeq{
\ipF<ρ(-φ) A ρ(φ),B> = \mathrm{Tr}(ρ(-φ) A^\star ρ(φ) B)
=  \mathrm{Tr}(ρ(-φ) A^\star ρ(φ) B ρ(-φ) ρ(φ)) \ccr
=  \mathrm{Tr}( A^\star ρ(φ) B ρ(-φ)) = \ipF<A,ρ(φ) B ρ(-φ)>.
}
Hence, for any matrix polynomials $A,B$,
\meeq{
\iip<ρ(-φ) \calR_φ A ρ(φ),B>_w
= \iint_Ω \ipF<ρ(-φ)A(ρ(φ)\bfx)ρ(φ),B(\bfx)>  w(\bfx) \D \bfx \ccr
= \iint_Ω \ipF<A(ρ(φ)\bfx),ρ(φ) B(\bfx) ρ(-φ)>  w(\bfx) \D \bfx \ccr
= \iint_Ω \ipF<A(\bfx),ρ(φ) B(ρ(-φ)\bfx) ρ(-φ)>  w(ρ(-φ)\bfx) \D \bfx\ccr
=\iip<A,ρ(φ) \calR_{-φ}  B ρ(-φ)>_w.
}
Thus we have, for matrix symmetry-adapted polynomials $A,B$ with respect to $m,n$ and all $φ$,
\meeq{
\E^{\I m φ} \iip<A, B>_w =  \iip<\E^{-\I m φ}A, B>_w =
\iip<ρ(φ)\calR_{-φ} A ρ(-φ), B>_w \ccr=
\iip<A, ρ(-φ)\calR_φ B ρ(φ)>_w =
 \iip<A, B>_w \E^{\I n φ}.
}
Therefore  $\iip<A, B>_w  = 0$ when $m ≠ n$.

\mqed

%


\subsection{Vector calculus}

An important feature of a symmetry-adapted function is that applying  differential operators
with certain rotational symmetries maintain the mode.  Here we consider  basic differential operators:
\begin{align*}
	\nabla &:= \Vectt[\partial_x,\partial_y] =   \bfe_r \partial_r + {1 \over r} \bfe_θ \partial_θ,&\qquad \hbox{(Gradient)} \\
	\nablaperp &:= \sopmatrix{0 & -1 \\ 1 & 0} \nabla = \Vectt[-\partial_y,\partial_x] =   \bfe_θ \partial_r - {1 \over r} \bfe_r \partial_θ,&\qquad \hbox{(Rotated Gradient)} \\
	\divergence &:= \nabla^\top = \vect[\partial_x, \partial_y] = {1\over r}\br[ \partial_r r \bfe_r^\top +  \partial_θ \bfe_θ^\top], &\qquad \hbox{(Divergence)} \\
	\curl &:= \nablaperp^\top = \vect[-\partial_y,\partial_x] = \divergence \underbrace{\sopmatrix{0 & 1 \\ -1 & 0}}_{ρ(-π/2)} = {1\over r} \br[ \partial_r r \bfe_θ^\top - \partial_θ \bfe_r^\top] & \qquad \hbox{(Curl)} \\
	Δ &:= \partial_{x}^2 + \partial_{y}^2 = \divergence\nabla= \partial_{r}^2 + {1 \over r} \partial_r + {1 \over r^2} \partial_θ^2, &\qquad \hbox{(Laplacian)}
\end{align*}
where $\partial_r = {x \partial_x + y \partial_y \over r}$ and $\partial_θ = x \partial_y -y \partial_x$. The formula for curl follows by transposing $ρ(π/2)  \bfe_r = \bfe_θ$ and  $ρ(π/2)  \bfe_θ = -\bfe_r$. To clarify the notation, note here that any vector-valued function $\bff : Ω \rightarrow \bbC^2$ can be written
\[
\bff(\bfx) =  g_1(\bfx) \bfe_r + g_2(\bfx) \bfe_θ
\]
where $\bfe_r^\top \bff = g_1$ and $\bfe_θ^\top \bff = g_2$.
Hence, for example, the divergence and curl formulae are given by
\[
\divergence \bff = {\partial_r[r g_1] +  \partial_θ g_2  \over r}, \qquad \curl \bff = {\partial_r[r g_2] -  \partial_θ g_1  \over r}.
\]

Each of these operators  intertwines with rotations in the following senses, where to avoid issues of regularity we assume the operations are acting on polynomials:

\Lemma{gradcommute} For all polynomials we have
\meeq{
	 \calR_φ \nabla = ρ(φ) \nabla \calR_φ, & \calR_φ \nablaperp & = ρ(φ) \nablaperp \calR_φ, \ccr
	 \calR_φ \divergence ρ(φ) = \divergence \calR_φ, & \calR_φ \curl ρ(φ) &= \curl \calR_φ,
}
Therefore $\calR_φ Δ = Δ \calR_φ$.

\begin{proof}
	Note that the polar coordinate partial derivatives commute with rotations:
\[
\partial_r \calR_φ = \calR_φ \partial_r, \qquad \partial_θ \calR_φ = \calR_φ \partial_θ.
\]
This combined with \eqref{Equation:rotateorthovecs} shows the result for the gradient and divergence:
\meeq{
\calR_φ \nabla  = \calR_φ\br[\bfe_r^{\bfx} \partial_r + {1 \over r} \bfe_θ^{\bfx} \partial_θ]
= \br[\bfe_r^{ρ(φ)\bfx}  \partial_r + {1 \over r} \bfe_θ^{ρ(φ)\bfx}  \partial_θ] \calR_φ =  ρ(φ) \nabla \calR_φ, \ccr
\calR_φ \divergence =
{1\over r}\br[ \partial_r r (\bfe_r^{ρ(φ) \bfx})^\top +  \partial_θ (\bfe_θ^{ρ(φ) \bfx})^\top] \calR_φ = \divergence ρ(-φ) \calR_φ.
}
The formul\ae\ for rotated gradient and curl follow:
\meeq{
\calR_φ \nablaperp = ρ(π/2) \calR_φ \nabla = ρ(π/2+φ) \nabla   \calR_φ  = ρ(φ) \nablaperp \calR_φ, \ccr
\calR_φ  \curl = \calR_φ \divergence ρ(-π/2) = \divergence ρ(-π/2-φ) \calR_φ = \curl ρ(-φ) \calR_φ.
}
	Finally, the Laplacian follows from combining the divergence and gradient via:
\[
Δ \calR_φ = \divergence \nabla \calR_φ =  \divergence ρ(-φ) \calR_φ \nabla = \calR_φ Δ.
\]
\end{proof}

\Corollary{diffopssymmetryadapted} Each differential operator maps a symmetry-adapted polynomial
to another symmetry-adapted polynomial with the same mode.

\begin{proof}
	Suppose $p$ is a scalar  symmetry-adapted polynomial with mode $m$. Then $\nabla p$ satisfies:
\[
\calR_φ\nabla p = ρ(φ) \nabla \calR_φ p = ρ(φ) \nabla p \E^{\I m φ}.
\]
If $\bfp$ is a vector symmetry-adapted polynomial with mode $m$ then
\meeq{
\calR_φ \divergence \bfp = \calR_φ \divergence ρ(φ) ρ(-φ) \bfp
=  \divergence \calR_φ ρ(-φ) \bfp =
\divergence \bfp \E^{\I m φ},
}
with the exact same argument showing the curl case. Finally the Laplacian case comes from combining the gradient and divergence.
\end{proof}

The results  extend to the transpose of the Jacobian matrix which we denote:
\[
\nabla \bff^\top := \sopmatrix{ f_{1x} & f_{2x} \\f_{1y} & f_{2y}}.
\]

\Corollary{vectorgradientmode}
If $\bfp$ is a vector symmetry-adapted polynomial with mode $m$ then $\nabla \bfp^\top$ is a matrix  symmetry-adapted polynomial with the same mode.

\begin{proof}
	Using the formula for gradients in \lmref{gradcommute} we have
\meeq{
\calR_φ \nabla \bfp^\top ρ(φ) = ρ(φ) \nabla (ρ(-φ)\calR_φ\bfp)^\top   = ρ(φ) \nabla ( \bfp \E^{\I m φ} )^\top
 = ρ(φ) \nabla \bfp^\top \E^{\I m φ}.
}
\end{proof}

This will guarantee orthogonality when combined with \lmref{ortho}.


\Section{homo} Symmetry-adapted homogeneous polynomials.

The aim of this section is to construct symmetry-adapted bases for homogeneous scalar, vector, and matrix polynomials.
We will later see that they can be orthogonalised in closed form for simple ultraspherical-like weights on the disk in terms of Jacobi polynomials.
Our general approach is to relate values of functions at $\bfx = r \vectt[\cos θ, \sin θ]$ to the value rotated to $θ = 0$, that is
at $r \bfe_1 = \vectt[r,0]$.

\Subsection Scalar homogeneous polynomials.

We present the scalar case in a way that will generalise to the vector and matrix case. The formula for a symmetry-adapted function, $\rmy_m(ρ(φ) \bfx)  = \rmy_m(\bfx) \E^{\I m φ}$,
 can  be used to deduce the harmonic polynomials from values on the $x$-axis:
\[
\rmy_m(\bfx)   = \rmy_m(r,0)  \E^{\I m φ} = r^{|m|} \E^{\I m φ}.
\]
Note that $r^{2k} \rmy_m(\bfx) $ are  symmetry-adapted homogeneous polynomials of degree $|m|+2k$.
Combining these  symmetry-adapted polynomials we can form a symmetry-adapted basis for all homogeneous polynomials
(and thence all polynomials):

\Proposition{scalarhomogeneous} A symmetry-adapted basis of degree $n$ homogeneous polynomials is given by, for $n$ even,
\[
	\rmy_{±n}, r^2 \rmy_{±(n-2)}, \ldots, r^{n-2} \rmy_{±2},  r^n \rmy_0,
\]
and for $n$ odd
\[
	\rmy_{±n}, r^2 \rmy_{±(n-2)}, \ldots,  r^{n-1} \rmy_{±1}.
\]

\Proof

We need to show this is a basis, i.e.,  they are linearly indepedent and their spans have the correct dimension.
Note there are exactly $(n+1)$-polynomials of degree $n$.
As they all have different modes they are orthogonal, by \lmref{ortho}.

\mqed

\Subsection Vector homogeneous polynomials.

The gradient of the basis of harmonic polynomials is
\begin{align*}
 \nabla \rmy_m &=  \nabla[r^{|m|} \E^{\I m θ}]  = |m| r^{|m|-1} ( \bfe_r + \I \sign m \, \bfe_θ) \E^{\I m θ}.
\end{align*}
From  \corref{diffopssymmetryadapted} we know that $\nabla \rmy_m$ is a vector symmetry-adapted function with mode $m$,
that is it satisfies:
\[
\calR_φ \nabla y_m(\bfx) = ρ(φ) \nabla y_m(\bfx) \E^{\I m φ}.
\]
 As in the scalar case, this formula  gives an expression in terms of the value at $θ = 0$:
\[
\nabla \rmy_m(\bfx)=  ρ(θ)\nabla \rmy_m(r,0) \E^{\I m θ}  = |m| r^{|m|-1}  ρ(θ) \Vectt[1,\I \sign m] \E^{\I m θ}.
\]
We use this to define the following vector analogues of a basis of harmonic polynomials, removing the multiplication by $m$ to extend the definition to $m = 0$:
\Definition{vectorharmonicpolys}
\begin{align*}
\bfy_0(\bfx) &:= r^{-1} ρ(θ) \Vectt[1,\I] = r^{-1}\pr( \bfe_r  + \I \bfe_θ ) = r^{-2} \sopmatrix{x - \I y \\ y + \I x}, \\
\bfy_m(\bfx) & := r^{|m|-1}  ρ(θ)  \Vectt[1,\I \sign m]   \E^{\I m θ}
= r^{|m|-1}  \pr( \bfe_r  + \I \sign m \bfe_θ ) \E^{\I m θ},\qquad  m≠ 0.
\end{align*}
For $m ≠ 0$ these are also homogeneous polynomials  of degree $|m|-1$ (since they are rescaled versions of $\nabla \rmy_m$ and partial derivatives map homogeneous polynomials to homogeneous polynomials).
For $m = 0$, where it is no longer the gradient of a harmonic polynomial, it is a degree 1 homogeneous polynomial divided by $r^2$.
Note  that constant vectors are {\it not} equivariant, they are spanned by vectors with mode $±1$:
\[
\bfy_{±1}(\bfx) = \Vectt[1,±\I].
\]

Now $r^{2k} \bfy_m$ will also be a vector  symmetry-adapted polynomial with mode $m$ but  of degree $|m|+2k-1$ for $m ≠ 0$ or $k > 1$. However, these polynomials do not span all polynomials.
Using the matrix equivariant function $Σ$ introduced in \eqref{Equation:G} allows us to construct a complete basis of homogeneous polynomials. For example,
we have two linearly independent degree $1$ equivariant polynomials given by:
\[
r^2 \bfy_0 = \Vectt[x - \I y, y + \I x], \qquad Σ \bfy_0 = \Vectt[x + \I y, y - \I x].
\]
To show this style of modification gives us a complete basis of linearly independent polynomials for general $m$ we establish an orthogonality property:



\Lemma{rhoGortho}
If $\bfa, \bfb : Ω \rightarrow \bbC^2$  are vector symmetry-adapted polynomials with mode $m$
 and $\bfa(r,0) = f(r) \Vectt[1,±\I]$, $\bfb(r,0) = g(r)  \Vectt[1,±\I]$ then
$
\ip<\bfa, Σ \bfb>_w = 0
$.

\Proof
We find:
\meeq{
\ip<\bfa,Σ \bfb>_w =  \int _0^1\int_0^{2π} \bfa(r ρ(θ) \vc e_1)^\star  Σ(r ρ(θ) \vc e_1) \bfb(r ρ(θ) \vc e_1) \Dθ w(r) r \D r\ccr
= \int_0^{2π} \E^{-\I m θ} \int_0^1 \bfa(r,0)^\star  Σ(r,0)  \bfb(r,0) w(r) r \D r  \E^{\I m θ} \D θ \ccr
= 2π \int_0^1 \underbrace{\sopmatrix{1 & \mp \I}   \sopmatrix{ 1 \\& -1} \Vectt[1, ±\I]}_{=0}   f(r) g(r)  w(r) r^3 \D r = 0.
}

\mqed

We thus can construct a complete basis of homogeneous vector polynomials that are symmetry-adapted:

\Lemma{vectorhomogeneous} A symmetry-adapted basis of degree $n$ homogeneous vector polynomials is given by,
 for $n$ even,
\begin{align*}
	\bfy_{±(n+1)}, r^2 \bfy_{±(n-1)}, \ldots,  r^n \bfy_{±1}, \\
	Σ \bfy_{±(n-1)}, r^2 Σ \bfy_{±(n-3)}, \ldots,  r^{n-2} Σ \bfy_{±1}.
\end{align*}
and, for $n$ odd,
\begin{align*}
	\bfy_{±(n+1)}, r^2 \bfy_{±(n-1)}, \ldots, r^{n-1} \bfy_{±2},  r^{n+1} \bfy_0, \\
	Σ \bfy_{±(n-1)}, r^2 Σ \bfy_{±(n-3)}, \ldots,  r^{n-3} Σ \bfy_{±2}, r^{n-1}Σ \bfy_0,
\end{align*}
where again
\[
Σ(x,y) := \sopmatrix{
x^2-y^2 & 2xy\\
 2xy  & y^2-x^2
}.
\]

\Proof
We will show this is in fact a complete orthogonal basis of degree $n$ homogeneous polynomials. In each case there are $2(n+1)$ homogeneous polynomials and hence we only need to show orthogonality. We first note that $r^{2k}Σ \bfy_m$ is a vector symmetry-adapted function, as
\[
\calR_φ r^{2k}Σ \bfy_m =  r^{2k} ρ(φ)Σρ(φ)^\top  ρ(φ) \bfy_m \E^{\I m φ} =  r^{2k} ρ(φ)Σ\bfy_m \E^{\I m φ}.
\]
As in the scalar case, vector symmetry-adapted polynomials corresponding to different modes are orthogonal (\lmref{ortho}). The functions with the same degree and mode, $r^{2k} \bfy_m$ and $r^{2k-2}  Σ \bfy_m$, are also orthogonal by the previous lemma.

\mqed

\Subsection Matrix homogeneous polynomials.

We now consider the transpose of the Jacobian of the vector analogues of harmonic polynomials. We compute, for $m ≠ 0$,
\meeq{
\nabla\bfy_m^\top =(|m|-1) r^{|m|-2} \br[ \bfe_r \bfe_r^\top -\bfe_θ \bfe_θ^\top +\I \sign m \pr(\bfe_r \bfe_θ^\top  + \bfe_θ \bfe_r^\top ) ] \E^{\I m θ}
}
which  by \corref{vectorgradientmode} is a matrix symmetry-adapted polynomial with mode $m$.
As in the vector case we can relate values of these to those on the $x$-axis, where we have
\meeq{
\nabla\bfy_m^\top(r,0) = (|m|-1) r^{|m|-2} \sopmatrix{1 & \I \sign m \\ \I \sign m & -1 }.
}
Dividing by $|m|-1$ motivates the definition of a matrix-analogue of harmonic polynomials, with a natural extension to $m = 0$:
\Definition{matrixharmonicpolys}
\begin{align*}
\rmY_0(\bfx) &:= r^{-2} ρ(θ) \sopmatrix{1 & \I  \\  \I & -1 } ρ(-θ) =  \sopmatrix{
1 & \I \\ \I & -1} r^{-2} \E^{-2 \I θ},
 \\
\rmY_m(\bfx) &:=  r^{|m|-2} ρ(θ) \sopmatrix{1 & \I \sign m \\ \I \sign m & -1 } ρ(-θ) \E^{\I m θ} \\
&=
\sopmatrix{
1 & \I \sign m \\ \I \sign m & -1} r^{|m|-2}  \E^{\I (m - 2\sign m) θ}, \qquad m ≠ 0.
\end{align*}
These are degree $|m|-2$ homogeneous polynomials for $|m| ≥ 2$ by their relationship with the gradient. In particular, we have two constant functions, with mode $±2$:
\[
\rmY_{±2}(\bfx) =  \sopmatrix{1 & ±\I \\ ±\I  & -1}.
\]
 For $m = 0$ and $m =±1$ we note that it is a polynomial divided by $r^4$ and $r^2$, respectively:
\meeq{
\rmY_0(\bfx) =  \sopmatrix{
1 & \I \\ \I & -1} {(x-\I y)^2 \over r^4},\qquad
\rmY_{±1}(\bfx) = \sopmatrix{
1 & ±\I \\ ±\I & -1} {x\mp \I y \over r^2}.
}
We can multiply these by $r^{2k}$ to get families of homogeneous polynomials of degree $2k+|m|-2$ when $|m| ≥ 2$, $m = ±1$ and $k ≥ 1$, or $k ≥ 2$. We can also multiply by $Σ$ on the left and right, which will increase the degree by $2$ whilst maintaining the mode.
Note that
\meeq{
Σ(x,y) \sopmatrix{1 & ±\I \\
±\I & -1} = \sopmatrix{1 & ±\I \\
\mp\I & 1} (x ± \I y)^2,\qquad 
\sopmatrix{1 & ±\I \\
±\I & -1} Σ(x,y) = \sopmatrix{1 & \mp\I \\
±\I & 1} (x±\I y)^2, \\
Σ(x,y) \sopmatrix{1 & ±\I \\
±\I & -1} &Σ(x,y) = \sopmatrix{1 & \mp\I \\
\mp\I & -1} (x±\I y)^4.
}

We can establish orthogonality via a matrix analogue of \lmref{rhoGortho}:

\Lemma{rhoGorthomatr}
If $A$ and $B$ are matrix symmetry-adapted polynomials with the same mode $m$
 and $A(r,0) = f(r) \sopmatrix{1 & ± \I \\ ± \I & -1}$, $B(r,0) = g(r)  \sopmatrix{1 & ± \I \\ ± \I & -1}$ then
\[
\iip<A, Σ B>_w  = \iip<A,  BΣ>_w = \iip<A, Σ B Σ>_w
 = \iip<ΣA,  BΣ>_w   = \iip<ΣA,  ΣBΣ>_w   = 0.
\]

\Proof
Note that
\[
A(\bfx) = A(r \rho(θ) \bfe_1) = f(r) ρ(θ) \sopmatrix{ 1   & ±\I \\ ±\I & -1} ρ(-θ) \E^{\I m θ} =
f(r) \underbrace{\sopmatrix{1 & ±\I \\ ±\I &-1}}_{=:C} \E^{\I (m \mp 2) θ}.
\]
and similarly $B(\bfx) = g(r) C  \E^{\I (m \mp 2) θ}$. A direct calculation shows for $S := \mdiag(1,-1)$ that
\meeq{
0 =\ipF<C,SC> = \ipF<C,CS> = \ipF<C,SCS>  = \ipF<SC,CS> = \ipF<SC,SCS>.
}
The result follows by reducing the integrals using $Σ(r,0) = r^2 S$, e.g.,
\meeq{
\iip<A, Σ B>_w = \int_0^1 \int_0^{2\pi} \ipF<A(r \rho(θ) \bfe_1) , Σ(r \rho(θ) \bfe_1) B(r \rho(θ) \bfe_1)> w(r) r \D r \D\theta \ccr
= \int_0^1 \int_0^{2\pi} \ipF<f(r) C \E^{\I (m \mp 2)θ} , r^2 S g(r) C \E^{\I (m \mp 2)θ}> w(r) r \D r \D\theta \ccr
= 2π \int_0^1 \bar f(r) g(r) w(r) r^3 \D r \ipF< C,SC  > = 0.
}
\mqed

We can combine these to deduce a complete basis of homogeneous polynomials:

\Lemma{matrixhomogeneous} A symmetry-adapted basis of degree $n$ homogeneous matrix polynomials is,
for $n = 0$,
\begin{align*}
\underbrace{\sopmatrix{1 & -\I \\-\I & -1}}_{\rmY_{-2}}, \underbrace{\sopmatrix{1 & \I \\\I & -1}}_{\rmY_2}, \underbrace{\sopmatrix{1 &\I \\ -\I & 1}}_{Σ \rmY_0}, \underbrace{\sopmatrix{1 &-\I \\ \I & 1}}_{\rmY_0 Σ},
\end{align*}
for $n = 1$,
\[
\underbrace{\sopmatrix{1 & ±\I \\ ±\I & -1} (x±\I y)}_{\rmY_{±3}},
\underbrace{\sopmatrix{1 & ±\I \\ ±\I & -1} (x\mp\I y)}_{r^2 \rmY_{±1}},
\underbrace{\sopmatrix{1 & ±\I \\ \mp\I & 1} (x±\I y)}_{Σ \rmY_{±1}},
\underbrace{\sopmatrix{1 & \mp\I \\ ±\I & 1} (x±\I y)}_{ \rmY_{±1} Σ}
\]
 for $n = 2,4,\ldots$
\begin{align*}
	\rmY_{±(n+2)}, r^2 \rmY_{±n}, \ldots, r^n \rmY_{±2}, r^{n+2}  \rmY_0 , \\
	Σ \rmY_{±n}, r^2 Σ \rmY_{±(n-2)}, \ldots,  r^{n-2} Σ \rmY_{±2}, r^n Σ \rmY_0, \\
		 \rmY_{±n}Σ, r^2 \rmY_{±(n-2)} Σ, \ldots,  r^{n-2}  \rmY_{±2} Σ, r^n  \rmY_0 Σ, \\
		  Σ    \rmY_{±(n-2)} Σ, \ldots,		 r^{n-4} Σ    \rmY_{±2} Σ ,  r^{n-2} Σ    \rmY_0 Σ.
\end{align*}
and for $n = 3,5,\ldots$
\begin{align*}
	\rmY_{±(n+2)}, r^2 \rmY_{±n}, \ldots, r^{n+1}  \rmY_{±1}, \\
	Σ \rmY_{±n}, r^2 Σ \rmY_{±(n-2)}, \ldots,  r^{n-1} Σ \rmY_{±1}, \\
		 \rmY_{±n}Σ, r^2 \rmY_{±(n-2)} Σ, \ldots,  r^{n-1}  \rmY_{±1} Σ, \\
		  Σ    \rmY_{±(n-2)} Σ, \ldots  r^{n-3} Σ    \rmY_{±1} Σ.
\end{align*}

\Proof

For each $n$ we have the correct total of $4(n+1)$ functions so we need to only show linear independence. As in the vector case we know different choices of modes $m$ are automatically orthogonal with respect to any inner product  of the form $\iip<\cdot,\cdot>_w$ (\lmref{ortho}) thus we only need to show orthogonality for the same $n$ and $m$. This is a direct consequence of the previous lemma.

\mqed
.

\Section{ops} Vector and matrix analogues of Zernike polynomials.

We now turn our attention to orthogonal polynomials beginning with analogues of Zernike polynomials, that is,
we construct scalar, vector, and matrix symmetry-adapted orthogonal polynomials with respect to the weight $(1-r^2)^λ$.

\Subsection Scalar Zernike polynomials.

The generalised Zernike polynomials $\rmz_{mj}^{(λ)}$  defined in the introduction are  symmetry-adapted functions with mode $m$ that are orthogonal with respect to the inner product
\[
\ip<f,g>_λ := \iint_Ω \bar f(\bfx) g(\bfx) (1-r^2)^λ \D \bfx.
\]
Choosing the right ordering gives us a basis for degree $n$ polynomials:

\Lemma{ZernikeOrtho}
Symmetry-adapted OPs with respect to $(1-r^2)^λ$ of degree $n$ are given by, for $n$ even,
\begin{align*}
\rmz_{±n,0}^{(λ)}, \rmz_{±(n-2),1}^{(λ)}, \ldots,  \rmz_{±2,n/2-1}^{(λ)},\rmz_{0,n/2}^{(λ)}
\end{align*}
and for $n$ odd,
\begin{align*}
\rmz_{±n,0}^{(λ)}, \rmz_{±(n-2),1}^{(λ)}, \ldots, \rmz_{±1,(n-1)/2}^{s,(λ)}.
\end{align*}

\begin{proof}
This is standard (see \cite{DunklXu}) but we write it in a way analogous to the vector and matrix cases, in particular, using the
symmetry-adapted property. By \lmref{ortho}, $\rmz_{mj}^{(λ)}$ and $\rmz_{nk}^{(λ)}$ are orthogonal if $m ≠ n$. If $m = n$ and $k ≠ j$, we have
\meeq{
\ip<\rmz_{mk}^{(λ)}, \rmz_{mj}^{(λ)}>_λ =  \int _0^1\int_0^{2π} \bar \rmz_{mk}^{(λ)}(r ρ(θ) \vc e_1)  \rmz_{mj}^{(λ)}(r ρ(θ) \vc e_1)   \Dθ (1-r^2)^λ r \D r\ccr
= \int_0^{2π} \E^{-\I m θ} \int_0^1 \bar \rmz_{mk}^{(λ)}(r,0) \rmz_{mj}^{(λ)}(r,0) (1-r^2)^λ r \D r  \E^{\I m θ} \D θ \ccr
=  2π \int_0^1 P_k^{(λ,|m|)}(2r^2-1) P_j^{(λ,|m|)}(2r^2-1) (1-r^2)^λ  r^{2|m|+1}   \D r.
}
With the change of variables $2r^2-1 = τ$ (so that $r^2 = (τ+1)/2$ and $1-r^2 = (1-τ)/2$) we find
\begin{align*}
\int_0^1 & P_k^{(λ,|m|)}(2r^2-1) P_j^{(λ,|m|)}(2r^2-1) (1-r^2)^λ  r^{2|m|+1} \D r  \ccr
=  {1 \over 2^{λ+|m|+2}} \int_{-1}^1 P_k^{(λ,|m|)}(τ) P_j^{(λ,|m|)}(τ) (1-τ)^λ (1+τ)^{|m|}  \D τ = 0.
\end{align*}
\end{proof}

\begin{remark}
We can relate these to the notation in \cite{vasil2016tensor}: $Q_j^{λ,m}$ corresponds to a rescaled version of $\rmz_{mk}^{(λ)}$  whilst $\E^{\I m θ}  \langle λ, m, r|$ corresponds to a rescaled version of the (infinite) row-vector $\begin{pmatrix} \rmz_{m0}^{(λ)} & \rmz_{m0}^{(λ)} & \ldots \end{pmatrix}$.
\end{remark}

\subsection{Vector Zernike polynomials} \label{Section:VectorZernike}

We now extend the construction to vector orthogonal polynomials. A simple way to construct vector OPs is to represent each component by a scalar OP but this will not be symmetry-adapted. Instead, we want to take the appropriate linear combination of the symmetry-adapted homogeneous orthogonal polynomials to achieve orthogonality. The following does so:
\Definition{vectorZernike}
\begin{align*}
\bfz_{mj}^{(λ),1}(x,y) &:=   P_j^{(λ,|m|-1)} (2r^2-1)\bfy_m(x,y), \\
\bfz_{mj}^{(λ),2}(x,y) &:=    P_{j-1}^{(λ,|m|+1)}(2r^2-1) Σ(x,y)\bfy_m(x,y).
\end{align*}
where for $m=0$ we use the standard convention for Jacobi polynomials with parameters corresponding to non-integrable weights coming from the series expansion \cite[18.5.7]{DLMF}, see \cite[4.22.2]{szeg1939orthogonal}:
\begin{align*}
P_0^{(λ,-1)}(x) &:= 1, \qquad P_n^{(λ,-1)}(x) := (1+x) {n+λ \over 2n} P_{n-1}^{(λ,1)}(x),\qquad n ≠ 0.
\end{align*}

We can use these to construct a basis for degree $n$ vector orthogonal polynomials:

\Theorem{vectorOPs} Symmetry-adapted vector OPs  with respect to $(1-r^2)^λ$ of degree $n$ are  given by, for  $n$ even,
\begin{align*}
\bfz_{±(n+1),0}^{(λ),1},&\bfz_{±(n-1),1}^{(λ),1},\ldots, \bfz_{±3,(n-2)/2}^{(λ),1},\bfz_{±1,n/2}^{(λ),1}, \\
&\bfz_{±(n-1),1}^{(λ),2},\ldots, \bfz_{±3,(n-2)/2}^{(λ),2},\bfz_{±1,n/2}^{(λ),2},
\end{align*}
and, for $n$ odd,
\begin{align*}
\bfz_{±(n+1),0}^{(λ),1},&\bfz_{±(n-1),1}^{(λ),1},\ldots,\bfz_{±2,(n-1)/2}^{(λ),1}, \bfz_{0,(n+1)/2}^{(λ),1}, \\
&\bfz_{±(n-1),1}^{(λ),2},\ldots,\bfz_{±2,(n-1)/2}^{(λ),2},\bfz_{0,(n+1)/2}^{(λ),2}.
\end{align*}

\begin{proof}

We note there are precisely $2(n+1)$ polynomials of degree $n$ and hence we need to show orthogonality.
\lmref{ortho} and \lmref{rhoGortho} show orthogonality  apart from the same mode and same superscripts. We have
\meeq{
\ip<\bfz_{mk}^{(λ),ν}, \bfz_{mj}^{(λ),ν}>_λ =  \int _0^1\int_0^{2π} \bfz_{mk}^{(λ),ν}(r ρ(θ) \vc e_1)^\star  \bfz_{mj}^{(λ),ν}(r ρ(θ) \vc e_1)   \Dθ (1-r^2)^λ r \D r\ccr
= \int_0^{2π} \E^{-\I m θ} \underbrace{\int_0^1 \bfz_{mk}^{(λ),ν}(r,0)^\star \bfz_{mj}^{(λ),ν}(r,0) (1-r^2)^λ r \D r}_{=: σ_{mkj}^{(λ),ν}}  \E^{\I m θ}\D θ = 2π σ_{mkj}^{(λ),ν}.
}
We  find, using $\bfz_{mk}^{(λ),1}(r,0) = P_k^{(λ,|m|-1)}(2r^2-1) r^{|m|-1} \vectt[1,±\I]$, for
$m ≠ 0$ that
\meeq{
 σ_{mkj}^{(λ),1} = 2\int_0^1  P_k^{(λ,|m|-1)}(2r^2-1) P_j^{(λ,|m|-1)}(2r^2-1) (1-r^2)^λ r^{2|m|-1}   \D r \ccr
= {1 \over 2^{λ+|m|}} \int_{-1}^1  P_k^{(λ,|m|-1)}(τ) P_j^{(λ,|m|-1)}(τ)  (1-τ)^λ (τ+1)^{|m|-1}  \D τ  = 0,
}
when $k ≠ j$. When $m = 0$ we use, where $k,j ≥ 1$:
\begin{align*}
 \int_{-1}^1  &P_k^{(λ,-1)}(τ) P_j^{(λ,-1)}(τ)  (1-τ)^λ (τ+1)^{-1}   \D τ \ccr  =
 {(k+λ)(j+λ) \over 4 k j} \int_{-1}^1  P_{k-1}^{(λ,1)}(τ) P_{j-1}^{(λ,1)}(τ)   (1-τ)^λ (τ+1)  \D τ  = 0,
\end{align*}
when $k ≠ j$.

Using $Σ(r,0) \bfy_m(r,0)  = r^{|m|+1} \vectt[1,±\I]$ we find, where $k,j ≥ 1$,
\meeq{
 σ_{mkj}^{(λ),2} = 2\int_0^1  P_{k-1}^{(λ,|m|+1)}(2r^2-1) P_{j-1}^{(λ,|m|+1)}(2r^2-1) (1-r^2)^λ r^{2|m|+3}   \D r \ccr
=   {1 \over 2^{λ+|m|+2}} \int_{-1}^1  P_{k-1}^{(λ,|m|+1)}(τ) P_{j-1}^{(λ,|m|+1)}(τ)  (1-τ)^λ (τ+1)^{|m|+1}   \D τ  = 0.
}

\end{proof}

\begin{remark}
We can relate these to the notation in \cite{vasil2016tensor}: we have $\bfe_+ =  (\bfe_r + \I \bfe_θ)/\sqrt{2} =  r\bfy_0/\sqrt{2}$ and $\bfe_- = (\bfe_r - \I \bfe_θ)/\sqrt{2} = r^{-1}  Σ \bfy_0/\sqrt{2}$. For $m \geq 0$, $\E^{\I m θ}  \bfe_- \langle \lambda, m-1, r |$ corresponds to a rescaled version of   $\begin{pmatrix} \bfz_{m0}^{(λ),1} & \bfz_{m1}^{(λ),1} & \ldots \end{pmatrix}$ whilst  $\E^{\I m θ}  \bfe_+ \langle \lambda, m+1, r |$  corresponds to a rescaled version of   $\begin{pmatrix} \bfz_{m0}^{(λ),2} & \bfz_{m1}^{(λ),2} & \ldots \end{pmatrix}$.
\end{remark}

\Subsection Matrix Zernike polynomials.

We now construct matrix orthogonal polynomials for the weight $(1-r^2)^λ$:

\Definition{matrixZernike}
\begin{align*}
\rmZ_{mj}^{(λ),1} &:=   P_j^{(λ,|m|-2)}(2r^2-1)\rmY_m &\rmZ_{mj}^{(λ),2} &:=   P_{j-1}^{(λ,|m|)}(2r^2-1) Σ \rmY_m \\    \rmZ_{mj}^{(λ),3} &:=    P_{j-1}^{(λ,|m|)}(2r^2-1) \rmY_m  Σ, & \rmZ_{mj}^{(λ),4} &:=    P_{j-2}^{(λ,|m|+2)}(2r^2-1) Σ \rmY_m  Σ.
\end{align*}
where for $m=0$ we again use the standard convention coming from the series expansion \cite[18.5.7
]{DLMF}, which satisfies \cite[4.22.2]{szeg1939orthogonal}
\begin{align*}
P_0^{(λ,-2)}(x) &:= 1,\qquad
P_1^{(λ,-2)}(x) := {2 + λ + λ x \over 2},\\
 \qquad P_n^{(λ,-2)}(x) &:=  { (n+λ-1)(n+λ) \over 4 n(n-1)} (1+x)^2 P_{n-2}^{(λ,2)}(x),\qquad n > 1.
\end{align*}

Choosing the right parameters we can build a complete basis of symmetry-adapted matrix OPs:

\Theorem{ortho} Symmetry-adapted matrix OPs  with respect to $(1-r^2)^λ$ are, for $n =0$,
\[
\rmZ_{±2,0}^{(λ),1}, \rmZ_{0,1}^{(λ),2},  \rmZ_{0,1}^{(λ),3},
\]
for $n>0$ even,
\begin{align*}
\rmZ_{±(n+2),0}^{(λ),1},\rmZ_{±n,1}^{(λ),1},&\rmZ_{±(n-2),2}^{(λ),1},\ldots,\rmZ_{±2,n/2}^{(λ),1}, \rmZ_{0,(n+2)/2}^{(λ),1} \\
\rmZ_{±n,1}^{(λ),2},&\rmZ_{±(n-2),2}^{(λ),2},\ldots, \rmZ_{±2,n/2}^{(λ),2},\rmZ_{0,(n+2)/2}^{(λ),2}, \\
\rmZ_{±n,1}^{(λ),3},&\rmZ_{±(n-2),2}^{(λ),3},\ldots, \rmZ_{±2,n/2}^{(λ),3},\rmZ_{0,(n+2)/2}^{(λ),3}, \\
&\rmZ_{±(n-2),2}^{(λ),4},\ldots, \rmZ_{±2,n/2}^{(λ),4},\rmZ_{0,(n+2)/2}^{(λ),4},
\end{align*}
and, for $n$ odd,
\begin{align*}
\rmZ_{±(n+2),0}^{(λ),1},\rmZ_{±n,1}^{(λ),1},&\rmZ_{±(n-2),2}^{(λ),1},\ldots,\rmZ_{±3,(n-1)/2}^{(λ),1}, \rmZ_{±1,(n+1)/2}^{(λ),1} \\
\rmZ_{±n,1}^{(λ),2},&\rmZ_{±(n-2),2}^{(λ),2},\ldots, \rmZ_{±3,(n-1)/2}^{(λ),2},\rmZ_{±1,(n+1)/2}^{(λ),2}, \\
\rmZ_{±n,1}^{(λ),3},&\rmZ_{±(n-2),2}^{(λ),3},\ldots, \rmZ_{±3,(n-1)/2}^{(λ),3},\rmZ_{±1,(n+1)/2}^{(λ),3}, \\
&\rmZ_{±(n-2),2}^{(λ),4},\ldots, \rmZ_{±3,(n-1)/2}^{(λ),4},\rmZ_{±1,(n+1)/2}^{(λ),4}.
\end{align*}

\begin{proof}
We have precisely $4(n+1)$ polynomials of degree $n$ so we need to only show orthogonality.
We again only need to consider the bases corresponding to the same modes and superscripts using \lmref{ortho} and \lmref{rhoGorthomatr}.
Similar to the vector case we have
\meeq{
\iip<\rmZ_{mk}^{(λ),ν}, \rmZ_{mj}^{(λ),ν}>_λ =  \int _0^1\int_0^{2π} \ipF<\rmZ_{mk}^{(λ),ν}(r ρ(θ) \vc e_1), \rmZ_{mj}^{(λ),ν}(r ρ(θ) \vc e_1)>   \Dθ (1-r^2)^λ r \D r\ccr
= \int_0^{2π} \E^{-\I m θ} \underbrace{\int_0^1 \ipF<\rmZ_{mk}^{(λ),ν}(r,0), \rmZ_{mj}^{(λ),ν}(r,0)> (1-r^2)^λ r \D r}_{=: σ_{mkj}^{(λ),ν}}  \E^{\I m θ}\D θ = 2π σ_{mkj}^{(λ),ν}.
}
First consider $ν = 1$. Note that
\[
\rmZ_{mk}^{(λ),1}(r,0) = r^{|m|-2}P_k^{(λ,|m|-2)}(2r^2-1)  \underbrace{\sopmatrix{1 & ±\I \\ ±\I & -1}}_C
\]
and $\ipF<C,C> = 4$. If  $|m| ≥ 2$, or $|m| = 1$ and $k,j > 0$, or $m = 0$ and $k,j > 1$ we have
\meeq{
 σ_{mkj}^{(λ),1} = 4\int_0^1  P_k^{(λ,|m|-2)}(2r^2-1) P_j^{(λ,|m|-2)}(2r^2-1) (1-r^2)^λ r^{2|m|-3}   \D r \ccr
= {1 \over 2^{λ+|m|-2}} \int_{-1}^1  P_k^{(λ,|m|-2)}(τ) P_j^{(λ,|m|-2)}(τ)  (1-τ)^λ (τ+1)^{|m|-2}  \D τ  = 0,
}
when $k ≠ j$. Here, when $m = 0$ we use:
\begin{align*}
 \int_{-1}^1  &P_k^{(λ,-2)}(τ) P_j^{(λ,-2)}(τ) (1-τ)^λ  (τ+1)^{-2}   \D τ \ccr  =
  \hbox{const.} \int_{-1}^1  P_{k-2}^{(λ,2)}(τ) P_{j-2}^{(λ,2)}(τ) (1-τ)^λ (τ+1)^2    \D τ  = 0
\end{align*}
when $k ≠ j$. Now note when $ν = 2$ we have
\[
\rmZ_{mk}^{(λ),2}(r,0) = r^{|m|}P_{k-1}^{(λ,|m|)}(2r^2-1) S  C
\]
where $S = \mdiag(1,-1)$, and $\ipF<SC, SC> = 4$.
Then
\meeq{
 σ_{mkj}^{(λ),2} = 4\int_0^1  P_{k-1}^{(λ,|m|)}(2r^2-1) P_{j-1}^{(λ,|m|)}(2r^2-1) (1-r^2)^λ r^{2|m|+1}   \D r \ccr
= {1 \over 2^{λ+|m|}} \int_{-1}^1  P_{k-1}^{(λ,|m|)}(τ) P_{j-1}^{(λ,|m|)}(τ)  (1-τ)^λ (τ+1)^{|m|}  \D τ  = 0.
}
when $k ≠ j$. A similar argument shows $σ_{mkj}^{(λ),3} = 0$. Finally, when $ν = 4$ note that
\[
\rmZ_{mk}^{(λ),4}(r,0) = r^{|m|+2}P_{k-2}^{(λ,|m|+2)}(2r^2-1) S  C S
\]
but again  $\ipF<SCS, SCS> = 4$. Thus
\meeq{
 σ_{mkj}^{(λ),4} = 4\int_0^1  P_{k-2}^{(λ,|m|+2)}(2r^2-1) P_{j-2}^{(λ,|m|+2)}(2r^2-1) (1-r^2)^λ r^{2|m|+5}   \D r \ccr
= {1 \over 2^{λ+|m|+2}} \int_{-1}^1  P_{k-2}^{(λ,|m|+2)}(τ) P_{j-2}^{(λ,|m|+2)}(τ)  (1-τ)^λ (τ+1)^{|m|+2}  \D τ  = 0.
}

\end{proof}

\begin{remark}
For comparison, in the notation of \cite{vasil2016tensor} we have
\meeq{
\bfe_+ \bfe_+ \quad \hbox{which denotes}\quad \bfe_+ \bfe_+^\top  = \sopmatrix{ 1 & \I \\ \I & -1} {(x-\I y)^2 \over 2r^2} =  {r^2 Y_0 \over 2},   \ccr
\bfe_- \bfe_+ \quad \hbox{which denotes}\quad \bfe_- \bfe_+^\top  = \sopmatrix{ 1 & \I \\ -\I & 1} {1 \over 2} = {Σ Y_0  \over 2},\ccr
\bfe_+ \bfe_- \quad \hbox{which denotes}\quad \bfe_+ \bfe_-^\top  = \sopmatrix{ 1 & -\I \\ \I & 1} {1 \over 2} = {Y_0 Σ \over 2},\ccr
\bfe_- \bfe_- \quad \hbox{which denotes}\quad \bfe_- \bfe_-^\top  = \sopmatrix{ 1 & -\I \\ -\I & -1} {(x+\I y)^2 \over 2r^2} = {ΣY_0Σ \over 2r^2}.
}
We then have, for $m \geq 0$, $\E^{\I m θ}  \bfe_+ \bfe_+ \langle \lambda, m-2, r |$ corresponds to $\begin{pmatrix} \rmZ_{m0}^{(λ),1} & \rmZ_{m1}^{(λ),1} & \ldots \end{pmatrix}$,  $\E^{\I m θ}  \bfe_-  \bfe_+ \langle \lambda, m, r |$  corresponds to    $\begin{pmatrix} \rmZ_{m0}^{(λ),2} & \rmZ_{m1}^{(λ),2} & \ldots \end{pmatrix}$,  $\E^{\I m θ}  \bfe_+  \bfe_- \langle \lambda, m, r |$  corresponds to   $\begin{pmatrix} \rmZ_{m0}^{(λ),3} & \rmZ_{m1}^{(λ),3} & \ldots \end{pmatrix}$,
and finally
  $\E^{\I m θ}  \bfe_-  \bfe_- \langle \lambda, m+2, r |$  corresponds to  $\begin{pmatrix} \rmZ_{m0}^{(λ),4} & \rmZ_{m1}^{(λ),4} & \ldots \end{pmatrix}$.

\end{remark}

\Section{vectoruniOPs} Special class of univariate vector orthogonal polynomials.

To go beyond simple ultraspherical inner products and allow more general equivariant symmetric positive definite matrix
weights we will relate vector OPs on the disk to a special class of univariate vector OPs. We first define a set of diagonal
matrix and vector polynomials whose constant term in a monomial expansion has a special form:

\Definition{rings}
Denote the commutative ring of diagonal matrix polynomials whose constant term in a monomial expansion is the identity by:
\[
 \calD := \set{c_0 I + \sum_{k=1}^p \sopmatrix{c_k \\ & d_k} t^k : c_k,d_k \in \bbR } \subset \bbR[t]^{2 \times 2}.
\]
Denote the module over $\calD$ of vector polynomials whose constant term in a monomial expansion have the same entry in each component by:
\begin{equation}\label{Equation:calM}
\calM := \calD \Vectt[1,1] = \set{c_0 \Vectt[1,1] + \sum_{k=1}^p \Vectt[c_k,d_k] t^k : c_k,d_k \in \bbR} \subset \bbR[t]^2.
\end{equation}

We want to consider orthogonal polynomials in $\calM$ with respect to diagonal inner products of the form
\[
\ip<\bff,\bfg>_V := \int_0^1 \bff(t)^\top \underbrace{\sopmatrix{
α(t) + t β(t) \\ & α(t) - t β(t)
}}_{=: V(t)} \bfg(t) \dt
\]
where $V : [0,1] \rightarrow \bbR^{2 \times 2}$ is symmetric positive definite almost everywhere,
and for simplicity we assume $α$ and $β$ are polynomial, that is, $V \in \calD$. We will see in \secref{matrixweights} that we can relate symmetry-adapted orthogonal polynomials with respect to an equivariant matrix weight in the disk with mode $m ≠ 0$ to
 vector orthogonal polynomials in $\calM$.

Consider orthogonal polynomials in $\calM$ with respect to $V$  arising from orthogonalising graded polynomials with the following ordering of monomials:
\[
\Vectt[1,1], \Vectt[1,-1]t, \Vectt[1,1] t,  \Vectt[1,-1]t^2, \Vectt[1,1] t^2, \cdots.
\]
We will denote a family of such orthogonal polynomials by
\begin{equation}\label{Equation:OPsinM}
\bfp_0^{V,1}, \bfp_1^{V,2}, \bfp_1^{V,1},  \bfp_2^{V,2}, \bfp_2^{V,1}, \ldots.
\end{equation}
Here $\bfp_n^{V,1/2}$ are a basis of degree $n$ polynomials in $\calM$, and we place the $\bfp_n^{V,2}$ term first so that the total basis is interlacing the bases $\{\bfp_n^{V,1}\}$ (whose leading order behaviour according to the grading is proportional to $\vectt[1,1] t^n$)  and $\{\bfp_n^{V,2}\}$ (whose leading order behaviour is proportional to $\vectt[1,-1] t^n$).
Note that we are not imposing a specific normalisation constant in the following, though these will be uniquely defined if we impose that they are monic polynomials, that is, that the highest order coefficient according to the above ordering is $1$.
It is convenient to also work with degree $n$ polynomials grouped together as a matrix polynomial,  which we denote
\[
\bfP_n^V(t) := \begin{cases} \pr({\bfp_0^{V,1}(t)}) & n = 0 \\
				\bvect[\bfp_n^{V,2}(t),\bfp_n^{V,1}(t)] & n ≠ 0 \end{cases} \in \bbR^{2 \times \min(n+1,2)}.
\]

A basic feature of orthogonal polynomials is the existence of three-term recurrences, which in the case of orthonormal
polynomials correspond to Jacobi
matrices (symmetric tridiagonal
matrices). These extend to multivariate and vector orthogonal polynomials in the form of symmetric
block tridiagonal matrices. Note that the ring $\calD$ has two generators: $t$ and
$t S$ where $S = \mdiag(1,-1)$. We get two associated block three-term recurrences:

\Lemma{VecJacobi} For $V \in \calD$ that is symmetric positive define almost everywhere, the matrix polynomials $\bfP_n^V$ whose columns are a basis of degree $n$ vector OPs in $\calM$   have two block three-term recurrences:
\meeq{
	t\bfP_0^V =  \bfP_0^V A_0^1 + \bfP^V_1 B_0^1,  \ccr
t \bfP_n^V =   \bfP_{n-1}^V C_n^1 + \bfP_n^V A_n^1 + \bfP^V_{n+1} B_n^1, \qquad n > 0,  \ccr
t  S \bfP^V_0 =  \bfP^V_0 A_0^2 + \bfP^V_1 B_0^2, \ccr
t  S \bfP^V_n =   \bfP^V_{n-1} C_n^2 + \bfP^V_n A_n^2 + \bfP^V_{n+1} B_n^2, \qquad n >0.
}
where $S = \mdiag(1,-1)$, $A_0^α \in \bbR^{1 \times 1}$, $B_0^α \in \bbR^{2 \times 1}$, $C_0^α \in \bbR^{1 \times 2}$, and
 $A_n^α,B_n^α,C_n^α \in \bbR^{2 \times 2}$, $n > 0$.
When  $\bfP_n^V$ are orthonormal then $(B_n^α)^\top = C_{n+1}^α$ and $A_n^α = (A_n^α)^\top$.

\Proof

The proof is identical to the proof for three-term recurrence for univariate polynomials since $t$ and $t S$ are both self-adjoint operators that increase the polynomial degree by 1 and
commute with $V \in \calD$. In particular, for $\bff,\bfg \in \calM$ we have
\meeq{
	\ip<\bff, t S \bfg>_V = \int_0^1 \bff(t)^\top V(t)tS \bfg(t) \dt = \int_0^1 (t S\bff(t))^\top V(t) \bfg(t) \dt = \ip<t S\bff, \bfg>_V
}
and similarly $\ip<\bff, t\bfg>_V = \ip<t\bff, \bfg>_V$.
Note that any degree $n$ polynomial $\bff \in \calM$ can be expanded as
\[
\bff =   \sum_{k=0}^n \bfP_k^V \bfc_k
\]
where
\[
	\bfc_k = \ip<\bfP_k^V, \bfP_k^V>^{-1} \ip<\bfP_k^V, \bff> \in \bbR^{\min(k+1,2)}.
\]
Thus by expanding each column we can write
\meeq{
t \bfP^V_n =  \sum_{k=0}^{n+1} \bfP_k^V Γ_{kn}^1\qqand
t S \bfP^V_n =  \sum_{k=0}^{n+1} \bfP_k^V Γ_{kn}^2.
}
where
\meeq{
	Γ_{kn}^1 = \ip<\bfP_k^V, \bfP_k^V>^{-1} \ip<\bfP_k^V, t \bfP_n^V> \in \bbR^{\min(k+1,2) \times 2}, \ccr
	Γ_{kn}^2 = \ip<\bfP_k^V, \bfP_k^V>^{-1} \ip<\bfP_k^V, t S\bfP_n^V> \in \bbR^{\min(k+1,2) \times 2}.
}
For $k ≤ n-2$ we know
\[
\ip<\bfP_k^V, t \bfP_n^V> =  \bigl<\!\!\!\!\!\!\!\!\!\!\!\!\underbrace{t \bfP_k^V}_{\hbox{degree $k+1 < n$}}\!\!\!\!\!\!\!\!\!\!\!\!, \bfP_n^V \bigr> = 0
\]
and similarly $\ip<\bfP_k^V, t S \bfP_n^V> = 0$ as $\bfP_n^V$ is orthogonal to all lower degree polynomials in $\calM$.
The three-term recurrence then follows by writing $C_n^ν := Γ_{n-1,n}^ν$, $A_n^ν := Γ_{n,n}^ν$, $B_n^ν := Γ_{n+1,n}^ν$.

Finally, note that for matrix polynomials $F,G$ whose columns are in $\calM$ we have
\[
\ip<F,G>_V^\top = \pr({\int_Ω F(t)^\top V(t) G(t) \D \bfx})^\top = \int_Ω G(t)^\top V(t) F(t) \D \bfx = \ip<G,F>_V.
\]
Thus when $\bfP_n^V$ are orthonormal we have
\meeq{
B_n^2 = \ip<\bfP_{n+1}, t  S \bfP_n > = \ip< t  S\bfP_{n+1}, \bfP_n > = \ip<\bfP_n,t  S\bfP_{n+1} >^\top = (C_n^2)^\top,
}
and by the same logic $B_n^1 = (C_n^1)^\top$.

\mqed

Associated with the three-term recurrences are two block Jacobi-like operators:
\begin{align}
t \bvect[\bfP_0^V,  \bfP_1^V, …] &= \bvect[\bfP_0^V,  \bfP_1^V, …] \underbrace{
\sopmatrix{A_0^1 & C_1^1\\
B_0^1 & A_1^1 & C_2^1 \\
& B_1^1 & A_2^1 & \ddots \\
&&\ddots & \ddots
}
}_{=:T_1}, \label{Equation:JacobiOp} \ccr
t S \bvect[\bfP_0^V,  \bfP_1^V, …] = \bvect[\bfP_0^V,  \bfP_1^V, …] \underbrace{
\sopmatrix{A_0^2 & C_1^2\\
B_0^2 & A_1^2 & C_2^2 \\
& B_1^2 & A_2^2 & \ddots \\
&&\ddots & \ddots
}
}_{=:T_2}. \notag
\end{align}
Note that $T_1^2 = -T_2^2$, though we will not use this property.

In \cite{gautschi1970construction} orthogonal polynomials with respect to polynomial weight modifications were deduced from Cholesky factorisations
(this was generalised to rational modifications in \cite{gutleb2024polynomial}).
We will use the same procedure here, though to avoid the need for orthonormal polynomials we use an LU factorisation instead of a Cholesky factorisation. In particular for a weight modification $M \in \calD$ which we write as
\[
M(t) = \sopmatrix{\tilde α(t) + t \tilde β(t) \\
& \tilde α(t) - t \tilde β(t)},
\]
and $\bfP^V := \bvect[\bfP_0^V,  \bfP_1^V, …]$ we can represent multiplication by $M$ in terms of the Jacobi-like operators:
\[
M(t) \bfP^V = \bfP^V\br[ \tilde α(T_1) + T_2 \tilde  β(T_1)].
\]
We can use this to construct orthogonal polynomials by computing a LU factorisation:

\begin{theorem}\label{Theorem:LU}
Suppose we have an LU factorisation:
\[
\tilde α(T_1) + T_2 \tilde  β(T_1) = LU
\]
where $L$ and $U$ are invertible. 	Vector OPs with respect to $M V$ in $\calM$ are given by the columns of $\bfP^V U^{-1}$.

\end{theorem}

\begin{proof}
Consider orthonormal polynomials which are a diagonal rescaling of $\bfP^V$,  that is we write $\bfQ^V:=\bfP^V D$ for a diagonal matrix $D = \mdiag(\|\bfp_0^{V,1}\|_V^{-1}, \|\bfp_1^{V,2}\|_V^{-1}, \|\bfp_1^{V,1}\|_V^{-1}, \ldots)$.
We have
\[
M \bfQ^V = \bfQ^V  D^{-1} (\tilde α(T_1) + T_2 \tilde  β(T_1)) D.
\]
The following is a symmetric positive definite banded matrix\footnote{It is an infinite matrix but bandedness ensures that finite-dimensional Cholesky factorisation results still apply, see \cite{gutleb2024polynomial}.} and therefore has a Cholesky factorisation:
\[
D^{-1}( \tilde α(T_1) + T_2 \tilde  β(T_1)) D = R^\top R,
\]
i.e., $M \bfQ^V = \bfQ^V R^\top R$. We then have
\[
\ip<\bfQ^V R^{-1}, \bfQ^V R^{-1}>_{MV} = \ip<\bfQ^V R^{-1}, M \bfQ^V R^{-1}>_V
= \ip<\bfQ^V R^{-1}, \bfQ^V R^\top>_V
=  R^{-\top} R^\top =I,
\]
thus $\bfQ^V R^{-1}$ are orthonormal with respect to $MV$.

Note that
\[
\tilde α(T_1) + T_2 \tilde  β(T_1) = \underbrace{D R^\top}_{\hbox{Lower triangular}} \times \underbrace{R D^{-1}}_{\hbox{Upper triangular}}.
\]
But LU factorisations are unique up to diagonal scaling, thus we know there exists a diagonal matrix $\tilde D$ such that $U = \tilde D R D^{-1}$. Thus we have
\[
\bfP^V U^{-1} = \bfQ^V D^{-1} D R^{-1}  \tilde D^{-1} = \bfQ^V  R^{-1}  \tilde D^{-1},
\]
i.e., the columns of $\bfP^V U^{-1}$ are rescaled orthonormal polynomials and thus orthogonal.

\end{proof}

\Subsection Scaled identity weights.

We now consider some special cases, beginning with the case where $β = 0$, so that
$
V(t) = α(t) I.
$
In this case we can construct vector OPs in $\calM$ directly in terms of standard univariate OPs,
in particular we have the following for Jacobi weights:

\Proposition{jacobiWt}
\begin{align*}
\bfp_n^{(a,b),1}(t) &:= P_n^{(a,b)}(2t-1) \Vectt[1, 1], \\
\bfp_n^{(a,b),2}(t) &:= t P_{n-1}^{(a,b+2)}(2t-1) \Vectt[1,-1], \qquad n ≥ 1
\end{align*}
are OPs with respect $(1-t)^a t^b I$ in $\calM$.

\Proof
Follows from direct inspection: if $k ≠ j$ we have
\meeq{
\int_0^1 \bfp_k^{(a,b),1}(t)^\top \bfp_j^{(a,b),1}(t) (1-t)^a t^b \D t = 2
\int_0^1 P_k^{(a,b)}(2t-1)P_j^{(a,b)}(2t-1)(1-t)^a t^b \D t = 0, \ccr
\int_0^1 \bfp_k^{(a,b),2}(t)^\top \bfp_j^{(a,b),2}(t) (1-t)^a t^b \D t = 2
\int_0^1 P_{k-1}^{(a,b+2)}(2t-1)P_{j-1}^{(a,b+2)}(2t-1)(1-t)^a t^{b+2} \D t = 0, \ccr
\int_0^1 \underbrace{\bfp_k^{(a,b),1}(t)^\top \bfp_j^{(a,b),2}(t)}_{=0} (1-t)^a t^b \D t = 0.
}

\mqed

We will only use the case where $a = 0$.
For this special case we can deduce the three-term recurrences explicitly:

\Lemma{threetermJacobi}
For
\[
\bfP_n^{(a,b)} := \begin{cases} \bfp_0^{(a,b),1} & n = 0 \\ \bvect[\bfp_n^{(a,b),2}, \bfp_n^{(a,b),1}] & n ≥ 1 \end{cases}
\]
we have the two block-three term recurrences
\meeq{
	t \bfP_0^{(0,b)} =  \bfP_0^{(0,b)}A_0^{1,(0,b)} +  \bfP_1^{(0,b)}B_0^{1,(0,b)}, \ccr
	t \bfP_n^{(0,b)} =  \bfP_{n-1}^{(0,b)}C_n^{1,(0,b)} + \bfP_{n}^{(0,b)}A_n^{1,(0,b)} +  \bfP_{n+1}^{(0,b)}B_n^{1,(0,b)}, \ccr
	t S \bfP_0^{(0,b)} = \bfP_0^{(0,b)}A_0^{2,(0,b)} +  \bfP_1^{(0,b)}B_0^{2,(0,b)}, \ccr
	t S \bfP_n^{(0,b)} =  \bfP_{n-1}^{(0,b)}C_n^{2,(0,b)} + \bfP_{n}^{(0,b)}A_n^{2,(0,b)} +  \bfP_{n+1}^{(0,b)}B_n^{2,(0,b)},
}
where
\meeq{
A_0^{1,(0,b)} = {b+1 \over b+2}, \qquad B_0^{1,(0,b)} = {1 \over b+2} \Vectt[0,1], \qquad C_1^{1,(0,b)} = {b+1 \over (b+2)(b+3)} \sopmatrix{ 0 & 1}, \ccr
A_n^{1,(0,b)} = {1 \over (2n+b)(2n+b+2)} \sopmatrix{b^2+b(2n+3)+2(n^2+n+1)\\ & b(b+1)+2bn+2n(n+1)}, \ccr
B_n^{1,(0,b)} = {1 \over (2n+b+1)(2n+b+2)} \sopmatrix{n(n+b+2)\\&(n+1)(n+b+1)}, \ccr
C_n^{1,(0,b)} = {1 \over (2n+b)(2n+b+1)} \sopmatrix{(n-1)(n+b+1)\\ &  n (n+b)},  \ccr
A_0^{2,(0,b)} = 0, \qquad B_0^{2,(0,b)} = \Vectt[1,0], \qquad  C_1^{2,(0,b)} = {b+1 \over b+3 } \sopmatrix{ 1& 0}, \ccr
A_n^{2,(0,b)} = {2n(n+b+1) \over (2n+b)(2n+b+2)}\sopmatrix{0 & 1 \\ 1 & 0}, \ccr
B_n^{2,(0,b)} = {1 \over (2n+b+1)(2n+b+2)} \sopmatrix{0 &  (n+b+1)(n+b+2)\\ n(n+1)  & 0} , \ccr
C_n^{2,(0,b)} ={1  \over (2n+b)(2n+b+1)} \sopmatrix{0 & (n-1)n \\ (n+b)(n+b+1) & 0},
}
where $n >1$.

\Proof

The multiplication by $t$ recurrences follow from the standard three-term recurrence for Jacobi polynomials \cite[18.9.2]{DLMF}.  The remaining recurrences follow from the lowering and raising relationships \cite[(18.9.5–6)]{DLMF}.

\mqed


\Subsection A special Jacobi-like matrix weight.

We  consider a specific weight which will prove essential to constructing a basis of  polynomials which are normal at the boundary of the disk. In particular, consider a matrix analogue of a Jacobi weight:
\[
V^{(b)}(t) := \sopmatrix{1 \\ & 1-t} t^b
= \sopmatrix{α(t) + tβ(t)\\ &α(t)-tβ(t)} t^b
\]
for $α(t) = 1-t/2$ and $β(t) = 1/2$.
This weight will prove important because it vanishes only in the second component at $t = 1$.
This vanishing property will  correspond to the tangential component vanishing when we relate these orthogonal polynomials to their multivariate counterparts on the disk.

We claim that the following gives an explicit construction of  vector OPs in $\calM$ with respect to the weight $V^{(b)}$:

\Definition{normalop} Define
\begin{align*}
	\weirdrat &:=	10n^2 + n(13b+19)+ 4(b+1)(b+2), \\
	\bfv_n^{(b)}(t) &:= \Vectt[(b+1)  (1-t)  P_n^{(1,b+1)}(2t-1) -2t  (n+1)  P_n^{(0,b+2)}(2t-1),
	(b+1)  P_n^{(1,b+1)}(2t-1)], \\
\bfq_n^{(b),1}(t) &:= {2n+b+1 \over (n+1) \weirdrat} \br[{(2n+b+2)(2n+b+3) P_n^{(1,b)}(2t-1) \Vectt[1-t,  1]-(n+b+2) \bfv_n^{(b)}(t)}], \\
\bfq_{n}^{(b),2}(t) &:=
- {\bfv_{n-1}^{(b)}(t) \over n} - \bfq_n^{(b),1}(t).
\end{align*}

Note that the normalisation constant is chosen because it will  lead to very simple expressions relating these to the gradient of weighted Zernike polynomials. We first show these are indeed polynomials of degree $n$, and also compute the leading order constants:

\Proposition{leadingorder}
\meeq{
\bfq_n^{(b),1}(t) = {(2n+b+1)! \over n! (n+b+1)! \weirdrat} \Vectt [2n^2+(5b+7)n+2b(b+3)+4,2(2n+b+1)(2n+b+2)]  t^n + O(t^{n-1}), \ccr
\bfq_n^{(b),2}(t) =  {2 (2n+b+2)! (2n+b+1) \over n! (n+b+1)! \weirdrat}   \Vectt[1,-1] t^n + O(t^{n-1}).
}

\begin{proof}
From \cite[§18.3]{DLMF} we know
\[
P_n^{(a,b)}(x) = {(2n+a+b)! \over 2^n n! (n+a+b)!} \br[x^n + {n(a-b) \over 2n+a+b} x^{n-1} +  O(x^{n-2})]
\]
and it follows that
\[
\bfv_n^{(b)}(t) =  {(2n+b+2)! \over n! (n+b+2)!}  \br[{-\Vectt[2n+b+3,0] t^{n+1} + \Vectt[{(n+1)(n+b+2)(2n+b+1) \over 2n+b+2}, b+1] t^n + O(t^{n-1})}].
\]
The proposition then follows from the definition. In particular, the term of order $t^{n+1}$ in $\bfv_n^{(b)}$ cancels with that of $\Vectt[(1-t) P_n^{(1,b)}(2t-1),0]$, leaving a degree $n$ polynomial.

\end{proof}

To show these are indeed orthogonal with respect to $V^{(b)}$ we will deduce lowering and raising operators
between $\bfq_n^{(b),j}$ and the previously defined  $\bfp_n^{(0,b),j}$. These are inferred from simple recurrence relationships:

\Lemma{NormalRaising}
Using the convention that $\bfq_{-1}^{(b),1} = \bfq_{-1}^{(b),2} = \bfq_0^{(b),2} = 0$
we have:
\meeq{
  \bfp_n^{(0,b),2} =
 -
{(n-1)(n+b+1) \over 2(2n+b)(2n+b+1)} \bfq_{n-1}^{(b),2} + {2 (n+b)(n+b+1) \over (2n+b+1)^2} \bfq_{n-1}^{(b),1} \\ &\qquad
+ {n\weirdrat \over 2 (2n+b+1)^2 (2n+b+2)} \bfq_n^{(b),2},\ccr
 \bfp_n^{(0,b),1} =  {(n-1)n \over 2(2n+b)(2n+b+1)} \bfq_{n-1}^{(b),2} -{2n(n+b) \over (2n+b+1)^2} \bfq_{n-1}^{(b),1}
 \\&\qquad
+  {n (n+b+1)(6n+3b+5) \over 2 (2n+b+1)^2 (2n+b+2)} \bfq_n^{(b),2} + {2(n+b+1) \over 2n+b+1} \bfq_n^{(b),1}, \ccr
 \weirdrat \sopmatrix{1 \\ & 1-t} \bfq_n^{(b),2}  = \weirdrat \bfp_n^{(0,b),2} + (n+b+1)(6n+3b+5) \bfp_n^{(0,b),1}
 \\
&\qquad -(n+b+2)(2n+b+1) \bfp_{n+1}^{(0,b),2}+(n+1) (2n+b+1) \bfp_{n+1}^{(0,b),1}, \ccr
{\weirdrat \over 2n+b+1} \sopmatrix{1 \\ & 1-t} \bfq_n^{(b),1}  =
(2n+b+3) \bfp_n^{(0,b),1} + (n+b+2) \bfp_{n+1}^{(0,b),2}
- (n+1) \bfp_{n+1}^{(0,b),1}.
}

\Proof
The lemma follows from using  \cite[(18.9.5–6)]{DLMF} to expand both the left- and right-hand sides into the basis $P_n^{(0,b+1)}(2t-1)$.
\mqed

\lmref{NormalRaising} encodes the definition of lower/upper block bidiagonal operators $L^{(b)}/R^{(b)}$
so that, for
\[
\bfQ_n^{(b)} := \begin{cases} \bfq_0^{(b),1} & n = 0 \\ \bvect[\bfq_n^{(b),2}, \bfq_n^{(b),1}] & n ≥ 1 \end{cases},
\]
we have
\meeq{
\bvect[\bfP_0^{(0,b)}, \bfP_1^{(0,b)},\ldots] = \bvect[\bfQ_0^{(b)}, \bfQ_1^{(b)},\ldots] \underbrace{ \sopmatrix{R_{00}^{(b),0} & R_{01}^{(b),1} \\
  & R_{11}^{(b),0} & \ddots \\
  & &\ddots}}_{=:R^{(b)}} \ccr
\sopmatrix{1 \\ & 1-t}  \bvect[\bfQ_0^{(b)}, \bfQ_1^{(b)},\ldots] = \bvect[\bfP_0^{(0,b)}, \bfP_1^{(0,b)},\ldots]
\underbrace{\sopmatrix{L_{00}^{(b),0} \\ L_{10}^{(b),1} & L_{11}^{(b),0} \\ &\ddots & \ddots}}_{ =:L^{(b)} }.
}
In particular, the recurrences tell us the entries of the blocks:
\meeq{
R_{00}^{(b)} = \Vectt[2], \qquad R_{01}^{(b)} = {2 (b+1) \over (b+3)^2}  \sopmatrix{b+2 & -1}, \ccr
R_{nn}^{(b)} = \sopmatrix{{n \over 2 (2n+b+1)^2 (2n+b+2)} \\ &  {2 (n+b+1) \over 2n+b+1}}   \sopmatrix{\weirdrat & (n+b+1)(6n+3b+5) \\ & 1}, \ccr
R_{n,n+1}^{(b)}  = \sopmatrix{{n \over 2 (2n+b+2)(2n+b+3)} \\ & {2(n+b+1) \over (2n+b+3)^2}} \sopmatrix{-1 & 1 \\
	1 & -1} \sopmatrix{n+b+2 \\ & n+1}, \ccr
L_{00}^{(b)} = {b+3 \over 4(b+2)} \Vectt[1],\qquad L_{10}^{(b)} = {1 \over 4 (b+2)} \Vectt[b+2,-1], \ccr
L_{nn}^{(b)} = \sopmatrix{1 \\ & {1 \over \weirdrat}} \sopmatrix{1 \\ (n+b+1)(6n+3b+5) & (2n+b+1)(2n+b+3)}, \ccr
L_{n+1,n}^{(b)} = {2n+b+1 \over \weirdrat} \sopmatrix{n+b+2 \\ & n+1} \sopmatrix{-1 & 1 \\ 1 & -1},
}
where $n > 0$.

We can relate the blocks of these raising and lowering operators to the blocks of the Jacobi-like operators of $\bfP_n$ deduced in \lmref{threetermJacobi} as follows:

\Lemma{ULcriteria}
For $n > 0$ we have
\meeq{
I + {A_0^{2,(0,b)} - A_0^{1,(0,b)} \over 2} = L_{00}^{(b)} R_{00}^{(b)}, \qquad
I + {A_n^{2,(0,b)} - A_n^{1,(0,b)} \over 2} = L_{nn}^{(b)} R_{nn}^{(b)} + L_{n,n-1}^{(b)} R_{n-1,n}^{(b)} \ccr
{B_{n-1}^{2,(0,b)}-B_{n-1}^{1,(0,b)} \over 2} =L_{n,n-1}^{(b)} R_{n-1,n-1}^{(b)} \qquad
{C_n^{2,(0,b)}-C_n^{1,(0,b)} \over 2} = L_{n-1,n-1}^{(b)} R_{n-1,n}^{(b)}.
}

\begin{proof}
	This follows from manipulating rationals.
\end{proof}

The above leads to a proof that $\bfq_n^{(b),ν}$ are indeed orthogonal polynomials in $\calM$ with respect to the weight $V^{(b)}$:

\Theorem{normalortho}
	$\bfq_n^{(b),ν}$ are orthogonal with respect $V^{(b)}(t) = \sopmatrix{1 \\ & 1-t} t^b$.

\begin{proof}
Using the Jacobi operators $T_ν$ in \eqref{Equation:JacobiOp} we can write
\meeq{
\sopmatrix{1 \\ & 1-t} \bvect[\bfP_0^{(0,b)}, \bfP_1^{(0,b)},\ldots] = ((1-t/2) I +  (1/2) t S) \bvect[\bfP_0^{(0,b)}, \bfP_1^{(0,b)},\ldots] \ccr
	=\bvect[\bfP_0^{(0,b)}, \bfP_1^{(0,b)},\ldots] \pr(I + {T_2- T_1 \over 2}).
}
The previous proposition tells us that we know its LU factorisation:
\meeq{
 I + {T_2- T_1 \over 2} = \sopmatrix{I + {A_0^{2,(0,b)} - A_0^{1,(0,b)} \over 2} & {C_1^{2,(0,b)}-C_1^{1,(0,b)} \over 2}\\
			{B_0^{2,(0,b)}-B_0^{1,(0,b)} \over 2} & I + {A_1^{2,(0,b)} - A_1^{1,(0,b)} \over 2}& {C_2^{2,(0,b)}-C_2^{1,(0,b)} \over 2} \\
			& {B_1^{2,(0,b)}-B_1^{2,(0,b)} \over 2} & I + {A_2^{2,(0,b)} - A_2^{1,(0,b)} \over 2}& \ddots \\
			&& \ddots & \ddots} \ccr
			 = \underbrace{\sopmatrix{L_{00}^{(b)} \\ L_{10}^{(b)} & L_{11}^{(b)}\\ & \ddots & \ddots}}_L \underbrace{\sopmatrix{R_{00}^{(b)} & R_{01}^{(b)} \\ & R_{11}^{(b)} & \ddots \\ && \ddots}}_R.
}
The theorem then follows from \thref{LU}.

\end{proof}



\Section{matrixweights} Symmetry-adapted vector OPs on the disk with equivariant matrix weights.

We can use  univariate vector OPs in $\calM$ to build symmetry-adapted vector OPs in the disk with respect to the inner product
\[
\ip<\bff,\bfg>_W := \iint_Ω \bff(x,y)^\star W(x,y) \bfg(x,y) \D \bfx
\]
where $W : Ω \rightarrow \bbR^{2 \times 2}$ is an equivariant, symmetric, and almost everywhere positive definite matrix polynomial.

We first establish a map between $\calM$ and the space of vector symmetry-adapted polynomials with mode $m$:



\Definition{modem}
Define the space of vector symmetry-adapted polynomials with mode $m$ as:
\begin{align*}
\Pi_0^2 &:= \mathspan\!\set{Σ \bfy_0, r^2 \bfy_0, r^2 Σ \bfy_0,…  }, \\
\Pi_m^2 &:= \mathspan\!\set{\bfy_m,Σ \bfy_m, r^2 \bfy_m, r^2 Σ \bfy_m,…  }, \qquad m ≠ 0.
\end{align*}

\Lemma{Pm}
For $m ≠0$,
\begin{align*}
\calP_m \bff(\bfx) &:= r^{|m|-1} \E^{\I m θ} ρ(θ) \sopmatrix{1 & \\  & \I \sign m} \bff(r^2)
\end{align*}
 is a one-to-one map $\calP_m : \calM \rightarrow \Pi_m^2$ with inverse
\[
\calP_m^{-1} \bff(t) =  t^{1-|m| \over 2} \sopmatrix{1 \\ & -\I \sign m} \bff(\sqrt{t},0).
\]
In particular, $\calP_m$ maps between polynomials of degree $n$ and polynomials of degree $2n+|m|-1$.

\begin{proof}
Recall the basis of
degree $2n+|m|-1$ homogeneous polynomials with $m$ introduced in
\lmref{vectorhomogeneous}: $r^{2n} \bfy_m(\bfx)$ and $r^{2n-2}Σ(\bfx) \bfy_m(\bfx)$.
Consider $\calP_m$ applied to a basis of $\calM$:
\meeq{
\calP_m\br[t^n {\Vectt[1,1]}](\bfx) = r^{2n+|m|-1} \E^{\I m θ} ρ(θ) \Vectt[1, \I \sign m] = r^{2n} \bfy_m(\bfx), \ccr
\calP_m\br[t^n {\Vectt[1,-1]}](\bfx) = r^{2n+|m|-1} \E^{\I m θ} ρ(θ) \Vectt[1, -\I \sign m] = r^{2n-2} Σ(𝐱) \bfy_m(\bfx), \quad n>0.
}
Thus we have a one-to-one map between expansions in these two bases.

We now  verify the inverse formula:
\meeq{
\calP_m^{-1}\br[r^{2n} \bfy_m](t)= t^{1-|m| \over 2 }  t^{2n+|m|-1 \over 2}  \Vectt[1,1] = t^n  \Vectt[1,1],  \ccr
\calP_m^{-1}\br[r^{2n-2} Σ \bfy_m](t) =  t^{1-|m| \over 2 }  t^{2n+|m|-1 \over 2}  \Vectt[1,-1] = t^n  \Vectt[1,-1], \quad n>0.
}

\end{proof}

We want to consider equivariant matrix polynomial weights $W$, which live in the matrix polynomial ring
\[
\Pi_0^{2 \times 2} := \mathspan\!\set{r^{2n} \rmY_0, r^{2n} Σ \rmY_0, r^{2n}  \rmY_0 Σ, r^{2n} Σ \rmY_0 Σ \quad \hbox{for}\quad n = 0,1,…}.
\]
We will translate the action of equivariant matrix polynomials on $Π_m^2$
to equivalent actions on $\calM$ via the following intertwining relationships:

\Proposition{intertwinematvec}
For $m ≠ 0$ we have:
\meeq{
\calP_m t^n  = r^{2n} \calP_m, \quad \calP_m  t S  = Σ \calP_m, \quad
\calP_m  t ρ(π/2)  =  \I \sign m ρ(π/2) Σ \calP_m,
}
where again $S = \sopmatrix{1 \\ & -1 }$ and $ρ(π/2) = \sopmatrix{0& -1 \\ 1 & 0}$.

\Proof
The first property is immediate since $t = r^2$. The second property follows from the equivariance of $Σ$, in particular:
\[
Σ(\bfx) ρ(θ) = ρ(θ)Σ(r,0) = r^2 ρ(θ) S.
\]
The last property follows since
\meeq{
ρ(π/2) Σ(\bfx) \calP_m \bff(\bfx) =
r^{|m|-1} \E^{\I m θ} ρ(θ) ρ(π/2)  \sopmatrix{1 & \\  & \I \sign m} S r^2 \bff(r^2) \ccr
= r^{|m|-1} \E^{\I m θ} ρ(θ)  \sopmatrix{0  & \I \sign m \\ 1 & 0} r^2 \bff(r^2)\ccr
= r^{|m|-1} \E^{\I m θ} ρ(θ)  \sopmatrix{- \I \sign m \\ & 1} ρ(π/2) r^2 \bff(r^2) \ccr
= -\I \sign m \calP_m[t ρ(π/2) \bff](\bfx).
}

\mqed

When we apply these relationships to equivariant weights on the disk we get a specific form of weight for $\calM$:

\Lemma{intertwineW} Suppose $W$ is a real symmetric equivariant  matrix polynomial. Then there exist $α,β,γ \in \bbR[t]$ such that
\[
W(x,y) = α(r^2) I + β(r^2) Σ(x,y) + γ(r^2) ρ(π/2) Σ(x,y).
\]
Moreover,
\meeq{
W(x,y) \calP_m = \calP_m V(t)
}
for
\[
V(t) = \sopmatrix{α(t) + t β(t)  & \I \sign m  t γ(t)   \\ -\I \sign m  t γ(t)  & α(t) - t β(t)}.
\]

\begin{proof}
	General equivariant matrix polynomials of (even) degree $p$ can be expanded in
	symmetric and skew-symmetric terms by recombining the basis $r^{2k+4} \rmY_0, r^{2k+2} Σ \rmY_0,r^{2k+2}  \rmY_0Σ$ and $r^{2k} Σ\rmY_0Σ$
	to write:
\meeq{
W(x,y) = \sum_{k=0}^{p/2} \br[α_k r^{2k} \underbrace{\sopmatrix{1 \\ & 1}}_{(Σ\rmY_0+\rmY_0Σ)/2} +  δ_k r^{2k} \underbrace{\sopmatrix{0 &-1 \\ 1 & 0}}_{(Σ\rmY_0-\rmY_0Σ)/(2\I)}] \\
&\qquad  + \sum_{k=0}^{p/2-1} \br[β_k r^{2k} \underbrace{\sopmatrix{x^2 - y^2 & 2xy \\ 2xy & y^2-x^2}}_{(r^4 \rmY_0 + Σ\rmY_0Σ)/2 = Σ(x,y)} + γ_k r^{2k}  \underbrace{\sopmatrix{-2xy & x^2-y^2 \\ x^2-y^2 & 2xy}}_{(r^4 \rmY_0 - Σ\rmY_0Σ)/(2\I) = ρ(π/2)Σ(x,y)}] \ccr
=  \underbrace{\sum_{k=0}^{p/2} α_k r^{2k}}_{=:α(r^2)} I + \underbrace{\sum_{k=0}^{p/2-1} β_k r^{2k}}_{=:β(r^2)} Σ(x,y) + \underbrace{\sum_{k=0}^{p/2-1} γ_k r^{2k}}_{=:γ(r^2)}ρ(π/2) Σ(x,y)
}
where we used $W(x,y) = W(x,y)^\top$ to deduce that $δ_k = 0$ for all $k$.
The lemma then follows from the intertwining relationships.
\end{proof}

We will from now on specialise on the case where $V$ is real, that is $γ = 0$ and hence
$V$ is also diagonal, matching the form considered in \secref{vectoruniOPs}. In this case we can construct multivariate vector OPs in terms of the univariate vector OPs in $\calM$
introduced in the previous section:

\Theorem{rhoGortho}
Suppose $\bfp^{V,(b),ν}_n := \bfp^{t^bV,ν}_n \in \calM$
are degree $n$ vector OPs in $\calM$ with respect to
\[
t^b \underbrace{\sopmatrix{α(t) +t β(t) \\ & α(t)-t β(t)}}_{V(t)}
\]
on $[0,1]$.
Further suppose $p_n^±(t)$ are degree $n$ univariate OPs with respect to $t(α(t)±tβ(t))$ on $[0,1]$.
Define the degree $|m| + 2j-1$ polynomials
\begin{align*}
\bfv_{0j}^{W,1}(\bfx) &:= p_{j-1}^-(r^2) r\bfe_θ, \qquad \bfv_{0j}^{W,2}(\bfx) := p_{j-1}^+(r^2) r\bfe_r, \qquad j ≥ 1,\\
\bfv^{W,ν}_{mj}(\bfx) &:= \calP_m \bfp^{V,(|m|-1),ν}_j(\bfx), \qquad m ≠ 0.
\end{align*}
Then symmetry-adapted vector OPs with respect to
\[
W(\bfx) = α(r^2) I + β(r^2) Σ(\bfx)
\]
of degree $n$ are given by, for $n$ even,
\begin{align*}
	\bfv_{±(n+1),0}^{W,1}, &\bfv_{±(n-1),1}^{W,1}, \ldots,\bfv_{±1,n/2}^{W,1}, \\
	&\bfv_{±(n-1),1}^{W,2}, \ldots, \bfv_{±1,n/2}^{W,2},
\end{align*}
and, for $n$ odd,
\begin{align*}
	\bfv_{±(n+1),0}^{W,1}, &\bfv_{±(n-1),1}^{W,1}, \ldots, \bfv_{±2,(n-1)/2}^{W,1}, \bfv_{0,(n+1)/2}^{W,1}, \\
	& \bfv_{±(n-1),1}^{W,2}, \ldots, \bfv_{±2,(n-1)/2}^{W,2},\bfv_{0,(n+1)/2}^{W,2}.
\end{align*}

\Proof
As in previous proofs, orthogonality between different modes is immediate and so we need only
show orthogonality for the same mode. Note we can reduce the inner product of the disk to an integral on the interval
using that $\bfv_{mj}^{W,ν}$ are symmetry-adapted and $W$ is equivariant:
\meeq{
\ip<\bfv^{W,ν}_{mk}, \bfv^{W,β}_{mj}>_W =  \int _0^1\int_0^{2π} \bfv^{W,ν}_{mk}(\bfx)^\star W(\bfx) \bfv^{W,β}_{mj}(\bfx)   \Dθ r \D r\ccr
= \int_0^{2π} \E^{-\I m θ} \int_0^1 \bfv^{W,ν}_{mk}(r,0)^\star ρ(-θ)ρ(θ)  W(r,0)  ρ(-θ)ρ(θ) \bfv^{W,β}_{mj}(r,0)r \D r \E^{\I m θ}  \D θ \ccr
= 2π \int_0^1 \bfv^{W,ν}_{mk}(r,0)^\star  V(r^2) \bfv^{W,β}_{mj}(r,0)r \D r,
}
using $W(r,0) = V(r^2)$.

For $m = 0$ we find:
\meeq{
\ip<\bfv^{W,1}_{0k}, \bfv^{W,2}_{0 j}>_W = 2π \int_0^1 p^+_{k-1}(r^2)  p^-_{j-1}(r^2) \bfe_1^\top V(r^2) \bfe_2 r^3 \D r = 0
}
since by assumption $V$ is diagonal. And for the same superscript we have, for $k ≠ j$,
\meeq{
\ip<\bfv^{W,1}_{0k}, \bfv^{W,1}_{0j}>_W = 2π \int_0^1 p^+_{k-1}(r^2)  p^+_{j-1}(r^2) \bfe_1^\top V(r^2) \bfe_1 r^3 \D r  \ccr
= π \int_0^1 p^+_{k-1}(t)  p^+_{j-1}(t) (α(t)+tβ(t)) t \D t = 0,
}
by the orthogonality of $p^+_k$. A similar argument shows $\ip<\bfv^{W,2}_{0k}, \bfv^{W,2}_{0j}>_W = 0$.

Now consider $m ≠ 0$.
We can write:
\meeq{
\bfv^{W,ν}_{mj}({r,0}) =\calP_m \bfp^{V,(|m|-1),ν}_j({r,0})
= r^{|m|-1} \underbrace{\mdiag(1,\I \sign m)}_{=:S_m} \bfp^{V,(|m|-1),ν}_j(r^2).
}
Using $ S_m^\star V(t) S_m = V(t)$ we find for $k ≠ j$:
\meeq{
\ip<\bfv^{W,ν}_{mk}, \bfv^{W,β}_{mj}>_W = 2π \int_0^1 \bfp^{V,(|m|-1),ν}_{mk}(r^2)^\top  S_m^\star  V(r^2) S_m \bfp^{V,(|m|-1),β}_{mj}(r^2)r^{2|m|-1} \D r \ccr
= π \int_0^1 \bfp^{V,(|m|-1),ν}_{mk}(t)^\top   V(t)  \bfp^{V,(|m|-1),β}_{mj}(t) t^{|m|-1} \D t = 0.
}

\mqed

In the special case of a uniform weight, $W(x,y) = I$ where  $α = 1$ and $β = 0$ and hence $t (α(t) ± t β(t)) = t$, we obtain via \propref{jacobiWt}:
\meeq{
\bfv_{0j}^{I,1}(\bfx) = P_{j-1}^{(0,1)}(2r^2-1) r \bfe_r, \qquad \bfv_{0j}^{I,2}(\bfx) = P_{j-1}^{(0,1)}(2r^2-1) r \bfe_θ, \qquad j ≥ 1, \ccr
\bfv_{mj}^{I,1}(\bfx) = \calP_m \bfp_j^{(0,|m|-1),1}(\bfx) = P_j^{(0,|m|-1)}(2r^2-1) r^{|m|-1} ρ(θ) \E^{\I m θ} \Vectt[1,\I \sign m] \ccr
		= P_j^{(0,|m|-1)}(2r^2-1) \bfy_m(\bfx), \qquad m ≠ 0\hbox{ and }j ≥ 0, \ccr
\bfv_{mj}^{I,2}(\bfx) = \calP_m \bfp_j^{(0,|m|-1),2}(\bfx) = P_{j-1}^{(1,|m|+1)}(2r^2-1) r^{|m|+1} ρ(θ) \E^{\I m θ} \Vectt[1,-\I \sign m] \ccr
= P_{j-1}^{(0,|m|+1)}(2r^2-1) Σ(\bfx) \bfy_m(\bfx), \qquad m ≠ 0\hbox{ and }j ≥ 1.
}
for $m ≠ 0$. We can relate these to the vector Zernike polynomials defined in \secref{VectorZernike} as follows:
\meeq{
\bfz_{0j}^1 = \bfv_{0j}^{I,1} + \I \bfv_{0j}^{I,2}, \qquad \bfz_{0j}^1 = \bfv_{0j}^{I,1} -\I \bfv_{0j}^{I,2}, \qquad  \bfz_{mj}^ν = \bfv_{mj}^{I,ν}.
}

\subsection{Orthogonal polynomials for a weight that is normal at the boundary}\label{Section:normal}

We can use the above construct to build $\bfn_{kj}^ν$, a natural basis for vector polynomials that are normal on the boundary of the disk. We introduce the symmetric-positive definite equivariant matrix weight:
\begin{equation}\label{Equation:N}
N(x,y) := \sopmatrix{
1 - y^2 & xy \\
x y & 1-x^2
} =  \pr(1-{r^2 \over 2}) I  +  {Σ(x,y) \over 2} = ρ(θ) \sopmatrix{1 \\ & 1 - r^2} ρ(-θ),
\end{equation}
noting that $ρ(θ) = \bvect[\bfe_r, \bfe_θ]$, that is, when $r = 0$ multiplication by $N$ imposes that the tangential component vanishes.
This weight is fundamental in describing polynomials that are normal at the boundary of a disk, i.e., in the following module over $\bbR[x,y]$:
\begin{equation}\label{Equation:calN}
\calN := \{ \bfp \in \bbR[x,y]^2 :  \bfp^\top \Vectt[-y,x] = 0 \hbox{ when } x^2+y^2 = 1\}.
\end{equation}
In particular, we will show that every $\bfp \in \calN$ can be written as $N$ times another vector polynomial.

We first decompose $\calN$ into ``true bubbles'' (polynomials that vanish at the boundary),
and two other simple polynomial sets:

\Lemma{normalpolyfirst}
\[
\bbR[x,y]^2 = \calK \oplus  \underbrace{\calE \oplus \calB}_{= \calN}
\]
for
\begin{align*}
\calK &:= \mathspan\!\set{\Vectt[1,0] , \Vectt[0,1], \Vectt[y,x], \Vectt[x,-y], x^k \Vectt[-y,x]\hbox{ for } k ≥ 0, x^{k-1} y \Vectt[-y,x] \hbox{ for } k > 0}, \\
\calE &:= \mathspan\!\set{x^k \Vectt[x,y]\hbox{ for } k ≥ 0, x^{k-1} y \Vectt[x,y] \hbox{ for } k > 0},  \\
\calB &:=  (1-x^2-y^2) \bbR[x,y]^2.
\end{align*}

\Proof

First note $\calE,\calB \subset \calN$. We will decompose homogeneous polynomials of each degree, and in the process show that
$\calK \cap \calN = \{0\}$ and $\calE \cap \calB$, which will imply  the decomposition. For degree 0 polynomials we have
\[
 \br[c	{\Vectt[1,0]} + d {\Vectt[0,1]}]^\top \Vectt[-y,x] = dx - cy
	\]
which by linear independence is only identically zero if $c = d = 0$, that is, no nontrivial degree 0 polynomial is in $\calN$. The space generated by degree 1 polynomials in $\calK$ satisfy, for $x^2+y^2 = 1$,
\[
\br[{c	\Vectt[-y,x] + d \Vectt[x,-y] + e \Vectt[y,x]}]^\top \Vectt[-y,x] = c- 2dxy + e (x^2-y^2) ≠ 0
\]
unless $c = d= 0$ by linear independence. The remaining basis element is $\Vectt[x,y] \in \calE \subset \calN$.
For higher order polynomials of degree $n≥2$ we only have two terms in $\calK$ which satisfy, for $x^2+y^2=1$,
\[
\Vectt[-y,x]^\top \Vectt[-y,x](c x^k  + d x^{k-1} y) = c x^k  + d x^{k-1} y ≠ 0
\]
unless $c = d = 0$. We then have $2n$ basis terms of degree $n$ given by:
\begin{align*}
\underbrace{x^{n-1} \Vectt[x,y] , x^{n-2} y \Vectt[x,y]}_{\in \calE},
 \underbrace{x^{n-2-k} y^k (1-x^2-y^2) \Vectt[1,0], x^{n-2-k} y^k (1-x^2-y^2) \Vectt[0,1]}_{\in \calB}
\end{align*}
for $k = 0,…,n-2$.
To see these are linearly independent we note for $x^2+y^2 = 1$ that the basis elements in $\calE$ satisfy
\[
\vectt[x,y] \br[c x^{n-1} {\Vectt[x,y]} + d x^{n-2} y {\Vectt[x,y]}] = c x^{n-1} + d x^{n-2} y ≠ 0
\]
unless $c = d = 0$, hence the intersection of their span with $\calB$ is trivial.

Since $\calB, \calE \subset \calN$ and $\calK \cap \calN = \{0\}$
we deduce that $\calN = \calE \oplus \calB$.

\mqed

\begin{remark}
	The definition of $\calK$ is not unique: it is one choice of basis and one can modify the basis of $\calK$ by adding elements of $\calN$ to form a new basis. For example, we could replace $x^k \Vectt[-y,x], x^{k-1} y\Vectt[-y,x]$ with the symmetry-adapted basis $(x±\I y)^k \Vectt[-y,x]$ or even $(x±\I y)^k \Vectt[y,x]$.
\end{remark}

A consequence of this lemma is that we can build a basis for $\calN$ by multiplying general vector polynomials by $N(x,y)$:

\Lemma{normalpoly}
$\calN =  \{ N \bfp : \bfp \in \bbR[x,y]^2\}$.

\Proof

Note that
\meeq{
\Vectt[x,y] = N(x,y) \Vectt[x,y],\quad (1-r^2) \Vectt[1,0] = N(x,y) \Vectt[1-x^2,-xy],\quad (1-r^2) \Vectt[0,1] = N(x,y) \Vectt[-xy,1-y^2].
}
The theorem follows since these generate bases for $\calE$ and $\calB$,
and hence $\calN$.

\mqed

It is therefore natural to consider a basis of vector polynomials which we multiple
by $N(x,y)$ to get a basis for $\calN$. Since $N$ is symmetric-positive definite we can
use it as a weight of orthogonality, and we will see
in \secref{veccalc} that these lead to simple formula for gradients and curl in terms of scalar (weighted) Zernike polynomials.
Using \thref{rhoGortho} alongside \thref{normalortho} we get an explicit representation:

\Corollary{normalops}
For $\bfq_j^{(b),ν}$ defined in \defref{normalop} we have
\meeq{
\bfv_{0j}^{N,1}(\bfx) =P_{j-1}^{(0,1)}(2r^2-1) r \bfe_r, \qquad \bfv_{0j}^{N,2}(\bfx) =P_{j-1}^{(1,1)}(2r^2-1) r \bfe_θ, \qquad j ≥ 1 \ccr
\bfv_{mj}^{N,ν}(\bfx) =\calP_m \bfq_j^{(|m|-1),ν}(\bfx) = r^{|m|-1} \E^{\I m θ}  ρ(θ) \sopmatrix{1 \\ & \I \sign m} \bfq_j^{(|m|-1),ν}(r^2),\qquad m ≠ 0.
}

We use these to define an explicit basis for $\calN$ as follows:

\Definition{normalops} Define polynomials of degree $≤ 2j+|m|+1$ by
\[
\bfn_{mj}^ν(\bfx) := N(\bfx) \bfv_{mj}^{N,ν}(\bfx).
\]

Note it will be helpful to  view these as orthogonal polynomials in $\calN$ with respect to the (non-integrable) weight $N^{-1}$,
that is, with respect to the inner product, for $\bff, \bfg \in \calN$,
\[
\ip<\bff,\bfg>_{N^{-1}} = \iint_Ω \bff(\bfx)^\star N^{-1}(\bfx) \bfg(\bfx)   \D \bfx
= \iint_Ω \bff(\bfx)^\star {T(\bfx) \over 1-x^2-y^2} \bfg(\bfx)   \D \bfx
\]
for a tangential analogue of $N$ defined by
\begin{equation}\label{Equation:T}
T(x,y) := \sopmatrix{
1 - x^2 & -xy \\
-x y & 1-y^2
} = ρ(θ) \sopmatrix{1 - r^2 \\ & 1} ρ(-θ),
\end{equation}
where we use $T(\bfx) N(\bfx) = N(\bfx) T(\bfx) = (1-x^2-y^2) I$.

\Section{veccalc}  Relationship with the de Rham complex of the disk.

We are now in a place to describe our main result.
We will show that $\{\rmw_{mj}\}, \{\bfn_{mj}^ν\},$ and $\{\rmz_{mj}\}$ give
natural bases for the de Rham complex in a disk, and in particular have very simple recurrence relationships relating their gradient and curl.
These lead to a decomposition of the de Rham complex into simple exact sub-complexes.

Recall the {\it de Rham complex with boundary conditions}  (see \cite[Section 4.5.5]{Arnold2018})  which in 2D is
\begin{align*}
0 \rightarrow  \Ho^1 \overset{\nabla}{\rightarrow}   \Hocurl \overset{\curl}{\rightarrow} L^2 \rightarrow 0,
\end{align*}
for the 2D Sobolev spaces on the disk:
\begin{align*}
	H^1 &:= \{f  : \| \nabla f\|^2 +  \|  f\|^2 < ∞ \}, \qquad \Hocurl :=  \{\bff \in \Hcurl : \| \curl \bff\|^2 +  \|  \bff\|^2 < ∞ \},\\
 \Ho^1 &:= \{f \in H^1 :  f|_{\partial Ω} = 0 \}, \qquad \Hocurl :=  \{\bff \in \Hcurl :  \bff|_{\partial Ω} \cdot \bfe_θ = 0 \},
\end{align*}
where the differential operators and trace operators are understood in the weak sense.

Polynomials that vanish on the boundary are a natural way to discretise $\Ho^1$ while
polynomials normal on the boundary are a natural way to discretise $\Hocurl$.
Thus $\{\rmw_{mj}\}$ serves as a natural basis for $\Ho^1$,
$\{\bfn_{mj}^ν\}$ is a natural basis for $\Hocurl$, and $\{\rmz_{mj}\}$ is a natural basis for $L^2$.
We will go one step further and recombine $\bfn_{mj}^ν$ to form a natural basis that
separates out the range of $\nabla$ / kernel of $\curl$:

\Definition{nplus}
\begin{align*}
\bfn_{0j}^+ &:= 2\bfn_{0j}^1, & \bfn_{0j}^- &:= 2\bfn_{0j}^2, \qquad j ≥ 1, \\
 \bfn_{m0}^- &:= 2\bfn_{m0}^1,&
\bfn_{mj}^± &:= \bfn_{mj}^1 ± \bfn_{mj}^2 \qquad m ≠ 0,\qquad j ≥ 1.
\end{align*}


Note that while the definition of $\bfn_{mj}^+$ is canonical,  the definition of $\bfn_{mj}^-$ is somewhat arbitrary, and the following discussion would work equally well with any linear combination of $\bfn_{mj}^1$ and $\bfn_{mj}^2$  that is linearly independent of $\bfn_{mj}^+$.
For example, an alternative linear combination could be chosen to preserve orthogonality with respect to $\ip<\cdot, \cdot>_{N^{-1}}$.

\subsection{Expressions for differential operators.}

We have extremely simple expression for the gradient relating the scalar and vector basis:

\Theorem{gradW}
\meeq{
\nabla \rmw_{mj} = -(j+1) \bfn_{m,j+1}^+.
}

\Proof
This can be shown using properties of Jacobi polynomials but we prefer an integration-by-parts argument that is more amenable to future generalisation. The gradient of a polynomial that vanishes on the circle is normal at the boundary, that is $\nabla \rmw_{mj} \in \calN$, and orthogonal to other modes (by \lmref{ortho}).  Hence we can expand
\[
\nabla \rmw_{mj} = \sum_{k=0}^∞ c_k^1 \bfn_{mk}^1 +
 \sum_{k=1}^∞ c_k^2 \bfn_{mk}^2,
\]
where by orthogonality we have
\[
c_k^ν = {\ip<\nabla \rmw_{mj},\bfn_{mk}^ν>_{N^{-1}} \over \ip<\bfn_{mk}^ν,\bfn_{mk}^ν>_{N^{-1}}},
\]
where we take $c_0^1 = 0$ when $m = 0$ (since $\bfn_{00}^1$ is not defined).
We first establish that $c_k^ν = 0$ unless $k = j+1$.
Recall that
$
N(x,y)^{-1} = (1- x^2-y^2)^{-1} T(x,y).
$
This tells us that if $\bff  =  N \bfg \in \calN$ then $\bfg = N^{-1} \bff$ is at most the same degree as $\bff$
(since multiplying by $T$ at most increases the degree by 2 while dividing by $1-x^2-y^2$ will decrease the degree by 2).

When $k \leq j$ we have by integration-by-parts (using that $\rmw_{mk}$ vanishes on the boundary):
\meeq{
	\ip<\nabla \rmw_{mj}, \bfn_{mk}^ν>_{N^{-1}} =
\big\langle\nabla \rmw_{mj}   ,\underbrace{N^{-1} \bfn_{mk}^ν}_{= \bfv_{mk}^{N,ν}}\big\rangle =
-
\bigl< \rmw_{mj}   ,\!\!\! \underbrace{\divergence \bfv_{mk}^{N,ν}}_{\mathrm{degree}\ ≤2k+|m|-2} \!\!\!\bigr> = 0
}
as $\rmw_{mj}$ is orthogonal to all polynomials of degree $< 2j+|m|$. On the other hand if $k > j+1$ we have
\meeq{
 \bigl< \!\!\!\underbrace{\nabla w_{mj}}_{\mathrm{degree}\ 2j+|m|+1}  \!\!\! , \bfn_{mk}^ν\bigr>_{ N^{-1}} =
 \bigl< \!\!\!\underbrace{N^{-1} \nabla w_{mj}}_{\mathrm{degree}\ ≤2j+|m|+1} \!\!\!,    \bfn_{mk}^ν\bigr> = 0
}
as $\bfn_{mk}^ν = N \bfv_{mk}^{N,ν}$ is orthogonal to polynomials of degree $ < 2k+|m|-1$. Therefore, we deduce that the expansion simplifies to:
\[
\nabla w_{mj} = c_{m,j+1}^1 \bfn_{m,j+1}^1 + c_{m,j+1}^2 \bfn_{m,j+1}^2.
\]

To deduce these constants we will match the leading order terms, beginning with the $m ≠ 0$ case. In particular, first note using \cite[§18.3]{DLMF}
\[
P_n^{(1,b)}(x) = {(2n+b+1)! \over 2^n n! (n+b+1)!} x^n + O(x^{n-1})
\]
implies that
\[
\rmw_{mj}(\bfx) =
-{(2j+|m|+1)! \over j! (j+|m|+1)!} r^{2j+|m| + 2}  \E^{\I m θ} + O(r^{2j+|m|+1}).
\]
We thus know the gradient is
\[
\nabla w_{mj}(\bfx) = -{(2j+|m|+1)! \over j! (j+|m|+1)!} r^{2j+|m|+1}  \br[(2j+|m|+2) \bfe_r + \I m \bfe_θ] \E^{\I m θ} + O(r^{2j+|m|})
\]
We multiply this by $T$ to find that
\[
T \nabla w_{mj}(\bfx) =
{(2j+|m|+2)! \over j! (j+|m|+1)!} r^{2j+|m|+3} \bfe_r \E^{\I m θ} + O(r^{2j+|m|}).
\]
On the other hand, note that
\[
\bfq_n^{(b),1}(t) + \bfq_n^{(b),2}(t) = {(2n+b+1)! \over n! (n+b+1)!} \Vectt [1,0]  t^n + O(t^{n-1}).
\]
If $\bff(t) = \bfc t^j + O(t^{j-1})$ then we have
\[
\calP_m \bff(\bfx) = r^{2j+|m| -1}  \bvect[\bfe_r, \bfe_θ] \sopmatrix{1 \\ &  \I \sign m}  \bfc  \E^{\I m θ} + O(r^{2j+|m| -3}).
\]
Therefore we have (using $T = (1-x^2-y^2) N^{-1}$)
\meeq{
T (\bfn_{m,j+1}^1(\bfx) + \bfn_{m,j+1}^2(\bfx)) = (1-r^2) \calP_m[\bfq_{j+1}^{|m|-1,1} + \bfq_{j+1}^{|m|-1,2}](\bfx) \ccr
= -{(2j+|m|+2)! \over (j+1)! (j+|m|+1)!}  r^{2j+|m|+3}  \bfe_r + O(r^{2j+|m|+1}).
}
Multiplying this by $-(j+1)$ matches the normalisation constant for $T \nabla w_{mj}$, that is
 $c_{mj}^1 = c_{mj}^2 = -(j+1)$. A similar matching term argument confirms the $m = 0$ formula.

%

\mqed

A similar argument gives a simple expression for the curl:

\Theorem{curlW} For $\weirdrat = 10n^2 + n(13b+19)+ 4(b+1)(b+2)$ we have $\curl \bfn_{mj}^+ = 0$ and
\meeq{
 \qquad \curl \bfn_{0j}^- = -4j \rmz_{0j}, \quad j ≥ 1, \ccr
\curl \bfn_{mj}^- = -{4(2j+|m|) (2j+|m|+1) (2j+|m|+2) \over 𝔯_j^{(|m|-1)}} \I \sign m  \rmz_{mj} \quad m ≠ 0.
}

\begin{proof}

We have $\curl \bfn_{mj}^+ = -\curl \nabla w_{m,j-1}/(j+1) = 0$.  By orthogonality to other modes (\lmref{ortho}) we know we can expand
\[
\curl \bfn_{mj}^ν = \sum_{k=0}^∞ c_k \rmz_{mk}, \qquad c_k = {\ip<\curl \bfn_{mj}^ν, \rmz_{mk}> \over \ip<\rmz_{mk},\rmz_{mk}>}.
\]
We first establish that $c_k = 0$ unless $k = j$. When $j < k$ we have
\[
\bigl<\underbrace{\curl \bfn_{mj}^ν}_{\mathrm{degree}\ ≤|m| + 2j}, \rmz_{mk}\bigr> =0
\]
since $\rmz_{mk}$ orthogonal to all polynomials of degree $< |m|+2k$.
When $j > k$ we integrate by parts:
\meeq{
\ip<\rmz_{mk},\curl \bfn_{mj}^ν> =
-\ip<\rmz_{mk},\divergence ρ(π/2)\bfn_{mj}^ν> \ccr
=-\oint_{\partial Ω} \rmz_{mk}  \underbrace{\bfx^\top ρ(π/2) \bfn_{mj}^ν}_{=0} \D s + \bigl<\nabla \rmz_{mk},
 ρ(π/2) \bfn_{mj}^ν\bigr> \ccr
   =\bigl<\underbrace{ρ(-π/2) \nabla \rmz_{mk}}_{\mathrm{degree}\ |m|+2k-1},
  \bfn_{mj}^ν\bigr>  =0
}
as $ \bfn_{mj}^ν = N \bfv_{mj}^{N,ν}$ are weighted (vector) orthogonal polynomials, therefore they are orthogonal to all polynomials of degree $< 2j+|m|-1$.

The   constant $c_j$ then follows by matching the leading order terms.

\end{proof}

\subsection{Exactness of the resulting de Rham sub-complexes.}

The exactness of the de Rham complex with boundary conditions is related to the {\it Betti numbers} (see eg. \cite{arnold2006finite}).
In the case of the disk the Betti numbers are $b_0 = 1$ (the number of connected components), $b_1 =0$ 
(the number of closed loops up to homotopy) and $b_2 = 0$ (the number of closed surfaces up to homotopy).
For a de Rham complex with boundary conditions these tell us that the dimension of the kernel of $\nabla$ is $b_2 = 0$,
the range of $\nabla$ equals the kernel of $\curl$ since the dimension of their quotient is $b_1 = 0$,
whilst there is a one dimension space that is not in the range of $\curl$ (which are the constant functions)
since the dimension of the quotient of the range of $\curl$ with $L^2$ equals $b_0 = 1$.

We can use the above results to decompose the space of polynomials into exact sub-complexes. In particular we have for each $m$ and $j \geq 0$ the following exact complexes:
\begin{align*}
0 \rightarrow  \mathspan\!\set{\rmw_{mj}} \overset{\nabla}{\rightarrow}   \mathspan\!\set{\bfn_{m,j+1}^+ \atop \bfn_{m,j+1}^-}   \overset{\curl}{\rightarrow} \mathspan\!\set{\rmz_{m,j+1}} \rightarrow 0,
\end{align*}
We can see this is exact as $\bfn_{mj}^+$ spans the kernel of $\curl$ as the curl of $\bfn_{mj}^-$ is not zero.  We also have  additional exact complexes for $m ≠ 0$:
\[
\begin{matrix}
0 &\rightarrow &   0 & \overset{\nabla}{\rightarrow}  &\mathspan\!\set{\bfn_{m0}^-}  & \overset{\curl}{\rightarrow} & \mathspan\!\set{\rmz_{m0}}  &\rightarrow &0.
\end{matrix}
\]
We finally have the constants $\mathspan\!\set{\rmz_{00}} \subset L^2$,  which are not in the range of $\curl$ and thus do not form part of an exact complex.  This matches the predicted theory since $b_0 = 1$.

\Section{cylinders} Extension to  cylinders.

The weighted vector orthogonal polynomials  $\bfn_{mj}^ν$ can be used to build bases for the 3D de Rham complex with boundary conditions in cylinders:
\begin{align*}
0 \rightarrow  \Ho^1 \overset{\nabla}{\rightarrow}   \Hocurl \overset{\curl}{\rightarrow}  \Hodiv \overset{\divergence}{\rightarrow}  L^2 \rightarrow 0,
\end{align*}
for the 3D Sobolev spaces over a cylinder $Γ$:
\begin{align*}
 \Hocurl &:=  \{\bff \in \Hcurl: \bff|_{\partial Ω}  \times \bfe_r = 0 \}, &
  \Hodiv &:=  \{\bff \in \Hdiv: \bff|_{\partial Ω}  \cdot \bfe_r = 0 \},
\end{align*}
where again differential and trace operators are understood in the weak sense.
Here we consider the case where $Γ$ is a periodic or finite cylinder:
\meeq{
  Ω \times \bbT = \{\vectt[x,y,z] : x^2+ y^2 ≤ 1, 0 ≤ z < 2π \},\ccr
  Ω \times [-1,1] = \{\vectt[x,y,z] : x^2+ y^2 ≤ 1, -1 ≤ z ≤ 1 \}
}
where $\bbT = [0,2π) \cong \bbR / (2π \bbR)$ is the periodic one-dimensional torus.

In order to represent a basis for $\Hodiv$ we will need to use a basis of vector polynomials that are tangential
on the boundary of a disk:
\begin{equation}\label{Equation:tangent}
\bft_{mj}^ν(\bfx) := ρ(π/2) \bfn_{mj}^ν(\bfx) =  T(\bfx)ρ(π/2) \bfv_{mj}^{N,ν}(\bfx).
\end{equation}
where multiplication by
\[
T(x,y) := \sopmatrix{
1 - x^2 & -xy \\
-x y & 1-y^2
} = ρ(θ) \sopmatrix{1 - r^2 \\ & 1} ρ(-θ)
\]
causes a vector to be tangential to the boundary, and we use the relationship that
\[
ρ(π/2) N(\bfx) = ρ(θ) \sopmatrix{0 & r^2 -1 \\ 1 & 0} ρ(-θ) = T(\bfx) ρ(π/2).
\]
Note that
$
\divergence \bft_{mj}^± = -\curl \bfn_{mj}^±
$
for $\bft_{mj}^± := ρ(π/2) \bfn_{mj}^±$.
Therefore we can adapt the simple formula in \thref{curlW}  to the divergence of $\bft_{mj}^±$, in particular $\divergence \bft_{mj}^+ = 0 $ whilst
\meeq{
\divergence \bft_{0j}^- = 4j \rmz_{0j}, \qquad j ≥ 1, \ccr
\divergence \bft_{mj}^- = {4(2j+|m|) (2j+|m|+1) (2j+|m|+2) \over 𝔯_j^{(|m|-1)}} \I \sign m  \rmz_{mj} \quad m ≠ 0.
}

\subsection{Cylinders with periodic boundary conditions}

We begin with the special case of a cylinder which is periodic in the $z$ direction, $Γ =Ω \times \bbT$. A natural basis for $\Ho^1$ is then
\[
\set{\E^{\I k z} \rmw_{mj}(x,y) : m,k \in \bbZ, j ≥ 0}.
\]
Note that this is a symmetry-adapted basis for both rotations (in $x,y$) and translations (in $z$, that is, a translation by $τ$ in the $z$ direction is equivalent to multiplication by $\E^{\I k τ}$).

A natural basis for $\Hocurl$ is
\begin{align*}
\set{\E^{\I k z} \Vectt[{\bfn_{mj}^+(x,y)}, 0] : m,k \in \bbZ, j ≥ 1  } &\cup \set{\E^{\I k z} \Vectt[{\bfn_{mj}^-(x,y)}, 0] : m,k \in \bbZ, j ≥ \max(0,1-|m|)  } \\
&\cup \set{\E^{\I k z} \Vectt[𝟎,{\rmw_{mj}(x,y)}]  : m,k \in \bbZ, j ≥ 0},
\end{align*}
which are normal on the boundary of the cylinder.
Indeed, the gradient of the basis for $\Ho^1$ satisfies a simple recurrence when expanded in the basis for $\Hocurl$:
\meeq{
\nabla \E^{\I k z} \rmw_{mj}(x,y,z) =   \E^{\I k z} \Vectt[{\nabla \rmw_{mj}(x,y)}, {\I k\rmw_{mj}(x,y)}] =   \E^{\I k z}  \Vectt[{-(j+1)\bfn_{m,j+1}^+(x,y)}, {\I k \rmw_{mj}(x,y)}] \ccr
= -(j+1) \E^{\I k z} \Vectt[{\bfn_{m,j+1}^+(x,y)}, 0] + \I k \E^{\I k z} \Vectt[𝟎,{\rmw_{mj}(x,y)}].
}

For $\Hodiv$ a natural basis is
\begin{align*}
\set{\E^{\I k z} \Vectt[{\bft_{mj}^+(x,y)}, 0] : m,k \in \bbZ, j ≥ 1  } &\cup \set{\E^{\I k z} \Vectt[{\bft_{mj}^-(x,y)}, 0] : m,k \in \bbZ, j ≥ \max(0,1-|m|)  } \\
&\cup \set{\E^{\I k z} \Vectt[𝟎,{\rmz_{mj}(x,y)}]  : m,k \in \bbZ, j ≥ 0},
\end{align*}
which is tangential on the boundary of the cylinder.
Indeed, writing the curl operator as
\[
\curl  \Vectt[
 {f},{g},{h}
] = \Vectt[\partial_y h - \partial_z g, \partial_z f - \partial_x h, \partial_x g - \partial_y f] = \Vectt[ρ(π/2) {\br[ \partial_z {\Vectt[f,g]}-{\Vectt[\partial_x,\partial_y]} h]}, \partial_x g - \partial_y f]
\]
we see that applying it to the $\Hocurl$ basis gives us a simple combination of the $\Hodiv$ basis:
\begin{align*}
\curl\ &\E^{\I k z} \Vectt[{\bfn_{mj}^+(x,y)}, 0] = \E^{\I k z} \Vectt[\I k ρ(π/2)  {\bfn_{mj}^+(x,y)},
\curl {\bfn_{mj}^+(x,y)}] =\I k \E^{\I k z} \Vectt[  {\bft_{mj}^+(x,y)}, 0], \\
\curl\ &\E^{\I k z} \Vectt[{\bfn_{mj}^-(x,y)}, 0] = \E^{\I k z} \Vectt[\I k ρ(π/2)  {\bfn_{mj}^-(x,y)},
\curl {\bfn_{mj}^-(x,y)}] \ccr
= \I k  \E^{\I k z} \Vectt[{\bft_{mj}^-(x,y)},0] -
κ_{mj} \E^{\I k z} \Vectt[𝟎,{\rmz_{mj}(x,y)}],\\
  \curl\ & \E^{\I k z} \Vectt[𝟎,{\rmw_{m,j-1}(x,y)}] = -\E^{\I k z} \Vectt[ρ(π/2) \nabla {\rmw_{m,j-1}(x,y)},0]
  = j \E^{\I k z}
  {\Vectt[ {\bft_{mj}^+(x,y)},0]}
\end{align*}
for
\begin{equation}\label{Equation:kappa}
κ_{mj} := {\begin{cases}
4j & m  = 0 \\
  {4(2j+|m|) (2j+|m|+1) (2j+|m|+2) \over 𝔯_j^{(|m|-1)}}  \I \sign m  & \hbox{otherwise}
  \end{cases}}.
\end{equation}

Finally, an orthogonal basis for $L^2$ is
\[
\set{ \E^{\I k z} \rmz_{mj}(x,y) : m,k \in \bbZ, j ≥ 0}
\]
and the divergence of the $\Hodiv$ basis has a simple relationship:
\meeq{
	\divergence \E^{\I k z} \Vectt[{\bft_{mj}^+(x,y)}, 0] =
\E^{\I k z} \divergence \bft_{mj}^+(x,y)
= -\E^{\I k z} \curl \bfn_{mj}^+(x,y) = 0 \ccr
\divergence \E^{\I k z} \Vectt[{\bft_{mj}^-(x,y)}, 0]
 =κ_{mj}  \E^{\I k z}  \rmz_{mj}(x,y), \ccr
\divergence \E^{\I k z} \Vectt[𝟎,{\rmz_{mj}(x,y)}] = \I k \E^{\I k z}  \rmz_{mj}(x,y).
}

\subsubsection{Exactness of the periodic cylinder de Rham complex with boundary conditions.}

We can use the above results to decompose the de Rham complex into sub-complexes. The periodic cylinder has Betti numbers
$b_0 = 1$ (one connected component), $b_1 = 1$ (exactly one loop up to homotopy), $b_2 = 0$ (all closed surfaces are homotopic to a point) and $b_3 = 0$. As in a disk, $b_0 = 1$ corresponds to the one-dimensional space of constants in $L^2$ which are not equal to the divergence of a in $\Hodiv$ (i.e., tangential to the boundary). Since $b_1 = 1$ we have a one-dimensional space
\[
\mathspan \underbrace{\Vectt[0,0,1]}_{=\Vectt[𝟎,\rmz_{00}]} \subset \Hodiv.
\]
which are in the kernel of $\divergence$ but cannot be written as the $\curl$ of a function in $\Hocurl$. 

Nevertheless, we can decompose the de Rham complex into the following exact sub-complexes for each $m$, $k$ and $j \geq 0$:
\begin{align*}
0 &\rightarrow  \mathspan\!\set{\E^{\I k z} \rmw_{mj}(x,y)} \overset{\nabla}{\rightarrow}   \mathspan\!\set{\begin{matrix} \E^{\I k z}\Vectt[{\bfn_{m,j+1}^+(x,y)},0]  \\ \E^{\I k z}\Vectt[{\bfn_{m,j+1}^-(x,y)},0]\\
\E^{\I k z} \Vectt[𝟎,{\rmw_{mj}(x,y)}]
\end{matrix}}  \\
& \overset{\curl}{\rightarrow}
\mathspan\!\set{\begin{matrix} \E^{\I k z}\Vectt[{\bft_{m,j+1}^+(x,y)},0]  \\ \E^{\I k z}\Vectt[{\bft_{m,j+1}^-(x,y)},0]\\
\E^{\I k z} \Vectt[𝟎,{\rmz_{m,j+1}(x,y)}]
\end{matrix}}  \overset{\divergence}{\rightarrow}
\mathspan\!\set{\E^{\I k z}\rmz_{m,j+1}(x,y)} \rightarrow 0.
\end{align*}
We can see this is exact as $\bfn_{mj}^+$ spans the kernel of $\curl$ as the curl of $\bfn_{mj}^-$ cannot be zero, and the only possible linear combination of the $\Hodiv$ basis whose divergence is zero is the curl of the $\Hocurl$ basis.   We also have  additional exact sub-complexes for $m ≠ 0$
\begin{align*}
0 \rightarrow   0 \overset{\nabla}{\rightarrow}  \mathspan\!\set{\E^{\I k z}  \Vectt[ {\bfn_{m0}^-(x,y)},0]}   \overset{\curl}{\rightarrow} & \mathspan\!\set{\begin{matrix} \E^{\I k z}\Vectt[{\bft_{m0}^+(x,y)},0]  \\ \E^{\I k z}\Vectt[{\bft_{m0}^-(x,y)},0] \\
\E^{\I k z} \Vectt[𝟎,{\rmz_{m0}(x,y)}]
\end{matrix}}
\overset{\divergence}{\rightarrow}
\mathspan\!\set{\E^{\I k z}\rmz_{m0}(x,y)} \rightarrow 0
\end{align*}
Finally, there are the additional  sub-complexes
\begin{align*}
0 \rightarrow   0 \overset{\nabla}{\rightarrow} 0  \overset{\curl}{\rightarrow} & \mathspan\!\set{
\E^{\I k z} \Vectt[𝟎,{\rmz_{00}(x,y)}]}
\overset{\divergence}{\rightarrow}
\mathspan\!\set{\E^{\I k z}\rmz_{00}(x,y)} \rightarrow 0.
\end{align*}
which are exact apart from when $k = 0$ (where we pick up the one-dimensional spaces predicted to not be part of an exact complex).

\subsection{Finite cylinders}

We can adapt the above construction to finite cylinders using suitable choices of (weighted) ultraspherical polynomials (or, equivalently, integrated Legendre functions \ala\  $p$-FEM \cite{babuska1981p}) instead of Fourier series in the $z$ direction. In particular, consider the cylinder $Γ = [-1,1] \times Ω$.  We use  the basis
\[
\set{(1-z^2) C_k^{(3/2)}(z) \rmw_{mj}(x,y) : m \in \bbZ, k ≥ 0, j ≥ 0}
\]
which vanishes on the boundary of the cylinder as a basis for $\Ho^1$, where $C_k^{(3/2)}(z) := {2 \over k+2} P_k^{(1,1)}(z)$, see \cite[18.7.1]{DLMF}. A natural basis  for $\Hocurl$ is then
\begin{align*}
&\set{(1-z^2) C_k^{(3/2)}(z)  \Vectt[{\bfn_{mj}^+(x,y)}, 0] : m \in \bbZ, k ≥ 0, j ≥ 1  } \\
 &\qquad\cup \set{(1-z^2) C_k^{(3/2)}(z) \Vectt[{\bfn_{mj}^-(x,y)}, 0] : m \in \bbZ, k ≥ 0 , j ≥ \max(0,1-|m|) }  \\
&\qquad\cup \set{P_k(z) \Vectt[𝟎,{\rmw_{mj}(x,y)}] : m \in \bbZ, k ≥ 0, j ≥ 0},
\end{align*}
which is normal at the boundary of the cylinder,
for the Legendre polynomials $P_n(z) = C_n^{(1/2)}(z) =  P_n^{(0,0)}(z)$.
Since, \cite[18.9.20]{DLMF},
\[
{ \D \over \dz}[(1-z^2) C_n^{(3/2)}(z)] =  -(n+1)(n+2) P_{n+1}(z)
\]
we have
\begin{align*}
\nabla &(1-z^2) C_k^{(3/2)}(z) \rmw_{mj}(x,y) \\
&= -(j+1)(1-z^2) C_k^{(3/2)}(z)  \Vectt[{\bfn_{m,j+1}^+(x,y)}, 0] -(k+1)(k+2) P_k(z) \Vectt[𝟎,{\rmw_{mj}(x,y)}].
\end{align*}
A natural basis for $\Hodiv$ is
\begin{align*}
&\set{P_k(z)  \Vectt[{\bft_{mj}^+(x,y)}, 0] : m \in \bbZ, k ≥ 0, j ≥ 1  } \\
 &\qquad\cup \set{P_k(z) \Vectt[{\bft_{mj}^-(x,y)}, 0] : m \in \bbZ, k ≥ 0 , j ≥ \max(0,1-|m|) }  \\
&\qquad\cup \set{(1-z^2) C_k^{(3/2)}(z) \Vectt[𝟎,{\rmz_{mj}(x,y)}] : m \in \bbZ, k ≥ 0, j ≥ 0},
\end{align*}
which are tangential on the boundary of a cylinder. We find that
\begin{align*}
\curl & (1-z^2)  C_k^{(3/2)}(z) \Vectt[{\bfn_{mj}^+(x,y)}, 0]
= -(k+1)(k+2) P_{k+1}(z) \Vectt[{\bft_{mj}^+(x,y)},0]\\
\curl & (1-z^2)  C_k^{(3/2)}(z) \Vectt[{\bfn_{mj}^-(x,y)}, 0]
= -(k+1)(k+2) P_{k+1}(z) \Vectt[{\bft_{mj}^-(x,y)},0]\\ & \qquad -
κ_{mj}(1-z^2) C_k^{(3/2)}(z)\Vectt[𝟎,{\rmz_{mj}(x,y)}],\\
  \curl& P_k(z) \Vectt[𝟎,{\rmw_{m,j-1}(x,y)}] = j P_k(z)
  \Vectt[ {\bft_{mj}^+(x,y)},0].
\end{align*}
Finally, a natural basis for $L^2$ is
\[
\set{P_k(z) \rmz_{mj}(x,y) : m \in \bbZ, k,j ≥ 0}
\] 
and indeed
\meeq{
\divergence P_k(z) \Vectt[{\bft_{mj}^+(x,y)}, 0] = 0,\ccr
\divergence P_k(z) \Vectt[{\bft_{mj}^-(x,y)}, 0] = κ_{mj} P_k(z)  \rmz_{mj}(x,y), \ccr
\divergence (1-z^2) C_k^{(3/2)}(z) \Vectt[𝟎,{\rmz_{mj}(x,y)}] = -(k+1)(k+2) P_{k+1}(z) \rmz_{mj}(x,y).
}

\subsubsection{Exactness of the finite cylinder de Rham complex with boundary conditions.}

For a finite cylinder the Betti numbers are $b_0 = 1$ whilst all other $b_k$ are zero. Thus like the disk we only have constants not forming part of an exact complex. We can decompose the de Rham complex into  exact sub-complexes are now of the form, for all $m$ and $k,j ≥ 0$:
\begin{align*}
0 &\rightarrow  \mathspan\!\set{(1-z^2)C_k^{(3/2)}(z) \rmw_{mj}(x,y)} \overset{\nabla}{\rightarrow}   \mathspan\!\set{\begin{matrix} (1-z^2) C_k(z) \Vectt[{\bfn_{m,j+1}^+(x,y)},0]  \\ (1-z^2) C_k(z) \Vectt[{\bfn_{m,j+1}^-(x,y)},0]\\
P_k(z) \Vectt[𝟎,{\rmw_{mj}(x,y)}]
\end{matrix}}  \\
& \overset{\curl}{\rightarrow}
\mathspan\!\set{\begin{matrix} P_{k+1}(z) \Vectt[{\bft_{m,j+1}^+(x,y)},0]  \\ P_{k+1}(z)  \Vectt[{\bft_{m,j+1}^-(x,y)},0]\\
(1-z^2) C_k^{(3/2)}(z)  \Vectt[𝟎,{\rmz_{m,j+1}(x,y)}]
\end{matrix}}  \overset{\divergence}{\rightarrow}
\mathspan\!\set{P_{k+1}(z) \rmz_{m,j+1}(x,y)} \rightarrow 0,
\end{align*}
We also have additional exact sub-complexes for $m ≠ 0$ and $ k ≥ 0$
\begin{align*}
0 &\rightarrow   0 \overset{\nabla}{\rightarrow}  \mathspan\!\set{(1-z^2) C_k(z)  \Vectt[ {\bfn_{m0}^-(x,y)},0]}  \\
& \overset{\curl}{\rightarrow}  \mathspan\!\set{\begin{matrix} P_{k+1}(z) \Vectt[{\bft_{m0}^-(x,y)},0] \\
(1-z^2)C_k(z) \Vectt[𝟎,{\rmz_{m0}(x,y)}]
\end{matrix}}
\overset{\divergence}{\rightarrow}
\mathspan\!\set{P_{k+1}(z) \rmz_{m0}(x,y)} \rightarrow 0
\end{align*}
Finally, we have two more additional exact sub-complex for $m ≠ 0$ and $j ≥ 0$:
\begin{align*}
0 \rightarrow   0 \overset{\nabla}{\rightarrow}  0  \overset{\curl}{\rightarrow} & \mathspan\!\set{\begin{matrix} \Vectt[{\bft_{m0}^-(x,y)},0]
\end{matrix}}
\overset{\divergence}{\rightarrow}
\mathspan\!\set{\rmz_{m0}(x,y)} \rightarrow 0, \\
0 \rightarrow   0 \overset{\nabla}{\rightarrow}  0  \overset{\curl}{\rightarrow} & \mathspan\!\set{\begin{matrix} (1-z^2) C_k^{(3/2)}(z)\Vectt[𝟎,{\rmz_{0j}(x,y) }]
\end{matrix}}
\overset{\divergence}{\rightarrow}
\mathspan\!\set{P_{k+1}(z) \rmz_{0j}(x,y)} \rightarrow 0.
\end{align*}
As predicted by the Betti numbers, we are left with the constants $\mathspan\!\set{P_0 \rmz_{00}} \subset L^2$,  which are not in the range of $\divergence$ and thus do not form part of an exact complex.

\Section{conc} Future directions.

We have constructed a new basis for vector polynomials that are normal at the boundary of a disk, which have simple recurrence relationships for the gradient and curl, corresponding to the 2D de Rham complex with boundary conditions. These immediately give a basis  for the 3D de Rham complex with boundary conditions on both a periodic and finite cylinder.
These results have applications in the numerical solution of PDEs in disks and cylinders a la the Finite Element Exterior Calculus (see, e.g., \cite{Arnold2018}), and will, for example,  lead to optimal complexity solvers for Maxwell-like equations in periodic cylinders, including equivariant variable coefficients.  We note that the mass matrices can be deduced from the recurrences in \lmref{NormalRaising}, which can be used to relate $\bfn_{mj}^ν$ to the Zernike vector polynomials $\bfz_{mj}^{(0),ν}$ that are orthogonal with respect to $L^2$.

A straightforward extension is to dropping the restriction on boundary conditions, which is necessary for incorporation into an $hp$-FEM framework \ala\ \cite{papadopoulos2025sparse}.  This would require adding the harmonic polynomials to $\rmw_{mj}$ to form a complete basis of  polynomials and defining $\calK$ in \lmref{normalpolyfirst} in terms of a suitable basis so that the recurrence relationships remain simple.  Extension to annuli is a more challenging endeavour. A scalar basis was used successfully for solving PDEs \cite{papadopoulos2024building,papadopoulos2025sparse} including in an $hp$-FEM framework, built on orthogonal polynomials with respect to the weight $(1-r^2)(r^2-ρ^2)$. The  construction in this paper was designed with extensions to the de Rham complex in annuli (and cylinderical annuli) in mind, however, it is unclear what  the analogue of $N(x,y)$ is on an annulus. It is possible one would need to view an analogue of $\bfn_{mj}^ν$ in terms of biorthogonality properties, \ala\ recent work on bases for $\Hcurl$ in hypercubes and simplices \cite{haubold2024high}.

Another possible extension is to higher-dimensional balls, in particular in 3D, which requires spherical harmonics instead of Fourier modes, involving matrix versions of symmetry-adapted bases.
On the surface of the sphere sparse recurrence relationships relating the surface gradient of spherical harmonics to spin weighted spherical harmonics has been used effectively in the numerical solution of PDEs \cite{vasil2019tensor}. What remains open is a similar construction on spherical caps, allowing for surface gradients and curls in the vein of \cite{snowball2021sparse}, though this would require working with tangent spaces. 

Symmetry-adapted bases can be used to construct equivariant bases on tensor product domains where the group action is applied simultaneously to each variable. For example, for $\bfx, \bfy \in Ω$ consider the matrix-valued polynomial
\[
N_{mkj}(\bfx,\bfy) := \bfn_{m,k}(\bfx) \bfn_{-m,j}(\bfy)^\top.
\]
This function is equivariant (similar to the sense of \defref{matrixequivariance}) when we apply the rotation to each variable:
\[
N_{mkj}(ρ(φ) \bfx, ρ(φ) \bfy)  ρ(φ)  = ρ(φ) \bfn_{m,k}(\bfx) \E^{\I m φ} \E^{-\I m φ} \bfn_{-m,j}(ρ(φ) \bfy)^\top  ρ(φ)^\top  ρ(φ)  =  ρ(φ) N_{mkj}(\bfx, \bfy).
\]
Thus we can use the orthogonal polynomials introduced here to construct equivariant orthogonal polynomials on such product domains.
Similar constructions in the 3D case, where the irreducible representations are given in terms of spherical harmonics, but not necessarily involving orthogonal polynomials, have proven effective in the Atomic Cluster Expansion in quantum chemistry \cite{bartok2013representing,grisafi2018symmetry,drautz2020atomic,dusson2022atomic,bigi2022smooth}.

We focussed on the de Rham complex but our construction is also applicable to the Koszul complex, see \cite[Section 7.2]{Arnold2018}, which in 2D has the form of multiplication by $\Vectt[x,y]$ and $\vectt[-y,x]$. In particular, an integration-by-parts argument shows that there exist constants such that
\meeq{
\Vectt[x,y] \rmz_{mj} = \hbox{const.} \bfn_{m,j+1}^1 +  \hbox{const.} \bfn_{m,j+1}^2 + \hbox{const.} \bfn_{m,j}^1 +  \hbox{const.} \bfn_{m,j}^2, \ccr
\vect[-y,x] \bfn_{mj} = \hbox{const.} \rmw_{m,j-1} + \hbox{const.} \rmw_{m,j-2}.
}
Matching leading order terms would likely give explicit expressions for these constants.

Also of interest is extension to the elasticity complex, cf. \cite[Section 8.8]{Arnold2018}, which would enable the efficient numerical solution of elasticity or Stokes flow in a cylinder. Note that whilst the approach of \cite{vasil2016tensor} works well for Stokes (and Navier–Stokes) in a cylinder, putting everything in the setting of the elasticity complex may facilitate discretisation of more general tensor fields. Other approaches are effective for Navier–Stokes in 2D \cite{torres1999pseudospectral,wilber2017computing,slomka2018stokes} but are based on a stream/vorticity formulation that breaks down in 3D.  More generally, there is recent work on a 2-complex for
matrix fields \cite{gopalakrishnan20252}, and it may be possible to construct matrix orthogonal polynomials in a disk corresponding to these matrix Sobolev spaces in a way that leads to simple recurrence relationships.

\bibliography{DiskRepOPs}

\ends